\documentclass[a4paper,leqno,tbtags]{amsart}
\usepackage{amsmath,amssymb,mathrsfs}
\usepackage{amscd,subfigure,graphicx,comment}

\usepackage{mathptmx,courier}
\usepackage[scaled=0.95]{helvet}

\usepackage[cmtip,arrow]{xy}
\usepackage{pb-diagram,pb-xy} % ,pb-lams,lamsarrow

\specialcomment{Notes}{}{} % {\begin{quote}\footnotesize}{\end{quote}}
\usepackage{ifpdf}
\ifpdf
\usepackage[pdftex,
            a4paper,
            pdfauthor={Michael Eisermann},
            pdftitle={Yang-Baxter deformations and rack cohomology},
            pdfcreator={LaTeX with hyperref},
            pdfproducer={LaTeX with hyperref},
            pdfstartview=FitH,
            bookmarksopen=true,
            bookmarksnumbered,
            colorlinks=true, 
            anchorcolor=black, 
            linkcolor=black, 
            urlcolor=black,
            citecolor=black, 
            pagecolor=black, 
            filecolor=black, 
            menucolor=black, 
            ]{hyperref}
\else
\usepackage[dvips,
            a4paper,
            pdfauthor={Michael Eisermann},
            pdftitle={Yang-Baxter deformations and rack cohomology},
            pdfcreator={LaTeX with hyperref},
            pdfproducer={LaTeX with hyperref},
            pdfstartview=FitH,
            bookmarksopen=true,
            bookmarksnumbered,
            colorlinks=true, 
            anchorcolor=black, 
            linkcolor=black, 
            urlcolor=black,
            citecolor=black, 
            pagecolor=black, 
            filecolor=black, 
            menucolor=black, 
            ]{hyperref}
\fi

\title{Yang-Baxter deformations and rack cohomology}

\author{Michael Eisermann}
\address{Institut Fourier, Universit\'e Grenoble I, France}
\email{Michael.Eisermann@ujf-grenoble.fr}
\urladdr{www-fourier.ujf-grenoble.fr/{\textasciitilde}eiserm}

\date{first version October 2007; this version compiled \today}

\numberwithin{equation}{section}

\theoremstyle{plain}
  \newtheorem{theorem}{Theorem}[section]
  \newtheorem{lemma}[theorem]{Lemma}
  \newtheorem{proposition}[theorem]{Proposition}
  \newtheorem{corollary}[theorem]{Corollary}
  
\theoremstyle{definition}
  \newtheorem{definition}[theorem]{Definition}
  \newtheorem{question}[theorem]{Question}
  \newtheorem{remark}[theorem]{Remark}
  \newtheorem{example}[theorem]{Example}
  \newtheorem*{notation}{Notation}

\newcommand{\sref}[1]{\textsection\ref{#1}}

\newcommand{\N}{\mathbb{N}}
\newcommand{\Z}{\mathbb{Z}}
\newcommand{\Q}{\mathbb{Q}}
\newcommand{\Zmod}[1]{\Z/_{\!\!#1}}

\newcommand{\K}{\mathbb{K}}                       % ground field
\newcommand{\A}{\mathbb{A}}                       % ring or algebra
\newcommand{\m}{\mathfrak{m}}                     % augmentation ideal
\newcommand{\fps}[1]{\mathopen{[\![}#1\mathclose{]\!]}}   % formal power series
\newcommand{\tensor}[1][]{\mathbin{\otimes_{#1}}} % tensor product [over a ring]

\newcommand{\Hom}{\operatorname{Hom}}             % set of homomorphisms
\newcommand{\End}{\operatorname{End}}             % set of endomorphisms
\newcommand{\tr}{\operatorname{tr}}               % trace

\newcommand{\Aut}{\operatorname{Aut}}             % set of automorphisms
\newcommand{\Inn}{\operatorname{Inn}}             % set of inner automorphisms
\newcommand{\tsa}{\mathbin{\overline{\ast}}}

\newcommand{\id}{\operatorname{id}}
\newcommand{\minus}{\smallsetminus}
\newcommand{\into}{\hookrightarrow}
\newcommand{\onto}{\mathrel{\makebox[1pt][l]{$\to$}{\to}}}
\newcommand{\isoto}{\mathrel{\xrightarrow{_\sim}}}% isomorphism

\newcommand{\YB}{{\smash{\scriptscriptstyle\rm YB}}}
\newcommand{\Rack}{{\smash{\scriptscriptstyle\rm R}}}
\newcommand{\Diag}{{\smash{\scriptscriptstyle\rm Diag}}}
\newcommand{\Ent}{E}
\newcommand{\indices}[2]%
{{\left[\begin{matrix}#1\\#2\end{matrix}\right]}} % large index array
\newcommand{\smallindices}[2]%
{{\left[\begin{smallmatrix}#1\\#2\end{smallmatrix}\right]}} % small index array
\newcommand{\pindices}[2]%
{{\left(\begin{matrix}#1\\#2\end{matrix}\right)}} % large index array
\newcommand{\smallpindices}[2]%
{{\left(\begin{smallmatrix}#1\\#2\end{smallmatrix}\right)}} % small index array

\graphicspath{{./figures/}}
\newlength{\imagewidth}
\newlength{\imageheight}
\newcommand{\picc}[2][1.0]{%
  \settowidth{\imagewidth}{\includegraphics[scale=#1]{#2}}
  \settoheight{\imageheight}{\includegraphics[scale=#1]{#2}}
  \parbox[c][\imageheight][c]{\imagewidth}{\includegraphics[scale=#1]{#2}}}

\begin{document} %%%%%%%%%%%%%%%%%%%%%%%%%%%%%%%%%%%%%%%%%%%%%%%%%%%%%%%%%%%%
%%%%%%%%%%%%%%%%%%%%%%%%%%%%%%%%%%%%%%%%%%%%%%%%%%%%%%%%%%%%%%%%%%%%%%%%%%%%%

\begin{abstract}
  Every rack $Q$ provides a set-theoretic solution $c_Q$ of the Yang-Baxter equation
  by setting $c_Q \colon x \tensor y \mapsto y \tensor x^y$ for all $x,y \in Q$.
  This article examines the deformation theory of $c_Q$ 
  within the space of Yang-Baxter operators over a ring $\A$, 
  a problem initiated by Freyd and Yetter in 1989.
  As our main result we classify deformations in the modular case, 
  which had previously been left in suspense, and establish that 
  every deformation of $c_Q$ is gauge-equivalent to a quasi-diagonal one.  
  Stated informally, in a quasi-diagonal deformation only behaviourally equivalent elements interact.
  In the extreme case, where all elements of $Q$ are behaviourally distinct, 
  Yang-Baxter cohomology thus collapses to its diagonal part,
  which we identify with rack cohomology.  
  % which has been intensely studied in recent years and is fairly well understood.
  % and which can now be naturally situated within Yang-Baxter theory.
  % has been intensely studied and now finds its natural place within Yang-Baxter theory.
  The latter has been intensively studied in recent years and, in the modular case,
  is known to produce non-trivial and topologically interesting Yang-Baxter deformations.
\end{abstract}

%% \copyrightinfo{2007}{Michael Eisermann}



%%%%%%%%%%%%%%%%%%%%%%%%%%%%%%%%%%%%%%%%%%%%%%%%%%%%%%%%%%%%%%%%%%%%%%%%%%%%%

% \headline{15mm}{\eprintinfo}

\vspace*{-5pt}

\maketitle

% \begin{quote} \small \tableofcontents \end{quote} \pagebreak[5]

%%%%%%%%%%%%%%%%%%%%%%%%%%%%%%%%%%%%%%%%%%%%%%%%%%%%%%%%%%%%%%%%%%%%%%%%%%%%%

\section{Introduction and statement of results} \label{sec:Introduction}

% This article considers set-theoretic solutions of the Yang-Baxter equation
% and studies their deformation theory.
% % In an attempt to be as non-technical as possible, 
% Before giving precise definitions in \sref{sec:Definitions},
% the aim of this introduction is to present a panoramic overview.

\subsection{Motivation and background}

Yang-Baxter operators (defined in \sref{sec:Definitions}) 
first appeared in theoretical physics, 
in a 1967 paper by Yang \cite{Yang:1967} on the many-body problem 
in one dimension, during the 1970s in work by Baxter \cite{Baxter:1972,Baxter:1982}
on exactly solvable models in statistical mechanics,
and later in quantum field theory \cite{Faddeev:1984}
in connection with the quantum inverse scattering method.
They have played a prominent r\^ole in knot theory and 
low-dimensional topology ever since the discovery of 
the Jones polynomial \cite{Jones:1987} in 1984. 
% \cite{Jones:1985,Jones:1987,Jones:1989}.
Attempts to systematically construct solutions of the Yang-Baxter equation 
have led to the theory of quantum groups, see Drinfel'd \cite{Drinfeld:1987} and 
Turaev, Kassel, Rosso \cite{Turaev:1988,Turaev:1994,Kassel:1995,KasselRossoTuraev:1997}.

All Yang-Baxter operators resulting from the quantum approach % so obtained 
are deformations of the transposition operator 
$\tau \colon x \tensor y \mapsto y \tensor x$.
As a consequence, the associated knot invariants are of finite type 
in the sense of Vassiliev \cite{Vassiliev:1990} and Gusarov \cite{Gusarov:1991}, 
see also Birmann--Lin \cite{BirmanLin:1993} and Bar-Natan \cite{BarNatan:1995}.
These invariants continue to have a profound impact on low-dimensional topology;
their interpretation in terms of algebraic topology and classical knot theory,
however, remains difficult and most often mysterious.

As a variation of the theme, Drinfel'd \cite{Drinfeld:1990} pointed out 
the interesting special case of \emph{set-theoretic solutions} 
of the Yang-Baxter equation, see Etingof--Schedler--Soloviev \cite{ESS:1999}
and Lu--Yan--Zhu \cite{LYZ:2000}.
One class of solutions is provided by \emph{racks} or \emph{automorphic sets}
$(Q,\ast)$, which have been introduced and thoroughly studied by 
Brieskorn \cite{Brieskorn:1988} in the context of braid group actions.  
Here the operator takes the form $c_Q \colon x \tensor y \mapsto y \tensor x^y$,
where $x^y = x \ast y$ denotes the action of the rack $Q$ on itself.
The transposition $\tau$ above corresponds to the trivial action;
conjugation $x^y = y^{-1} x y$ in a group provides many non-trivial examples.

Applications to knot theory had independently been developed 
by Joyce \cite{Joyce:1982} and Matveev \cite{Matveev:1982}. % in 1982.
Freyd and Yetter \cite{FreydYetter:1989} observed that the knot invariants 
obtained from $c_Q$ are the well-known colouring numbers of classical knot theory.
These invariants are not of finite type \cite{Eisermann:2000}.
% in the sense of Vassiliev-Gusarov \cite{Eisermann:2000}.
Freyd and Yetter \cite{FreydYetter:1989,Yetter:2001} also initiated 
the natural question of deforming set-theoretic solutions 
within the space of Yang-Baxter operators over a ring $\A$, 
and illustrated their general approach by the simplified ansatz 
of diagonal deformations \cite[\textsection 4]{FreydYetter:1989}.
The latter are encoded by rack cohomology, which was 
independently developed by Fenn and Rourke \cite{FennRourke:1992} 
from a homotopy-theoretic viewpoint via classifying spaces.

Carter et al.\ \cite{CarterEtAl:2003} % \cite{CarterEtAl:2001,CarterEtAl:2003,CarterEtAl:2005}
have applied rack and quandle cohomology to knots by constructing state-sum invariants.
% Applied to knots, these deformations have attracted much attention since 1999
% and have spawned an extensive literature where quandle cohomology 
% is used to construct state-sum invariants of knots \cite{CarterEtAl:2003}. 
These, in turn, can be interpreted in terms of classical algebraic topology 
as \emph{colouring polynomials} associated to knot group representations
\cite{Eisermann:2007}. 

% In an attempt to be as non-technical as possible, the purpose 
% of this introduction is to present a panoramic overview.
For more recent developments and open questions 
see \sref{sec:OpenQuestions} at the end of this article.

\subsection{Yang-Baxter deformations}

In the present article we continue the study of Yang-Baxter 
deformations of racks linearized over a ring $\A$, 
as initiated by Freyd and Yetter \cite{FreydYetter:1989}.
Detailed definitions will be given in \sref{sec:Definitions} 
below, in particular we will review Yang-Baxter operators
(\sref{sub:YangBaxterOperators}) and set-theoretic solutions
$c_Q$ coming from racks (\sref{sub:QuandlesAndRacks}).
In this introduction we merely recall the basic definitions 
in order to state our main result.

\begin{notation}[rings and modules]
  Throughout this article $\A$ denotes a commutative ring with unit.
  All modules will be $\A$-modules, all maps between modules will be 
  $\A$-linear, and all tensor products will be formed over $\A$.
  For every $\A$-module $V$ we denote by $V^{\tensor n}$ 
  the tensor product $V \tensor \cdots \tensor V$ of $n$ copies of $V$.
  Given a set $Q$ we denote by $\A{Q}$ the free $\A$-module with basis $Q$.  
  We identify the $n$-fold tensor product $(\A{Q})^{\tensor n}$ with $\A{Q^n}$.
  % Moreover, it will be convenient to identify $\A$-linear maps 
  In particular, this choice of bases allows us to identify $\A$-linear maps 
  $\A{Q^n} \to \A{Q^n}$ with matrices $Q^n \times Q^n \to \A$.

  For the purposes of deformation theory we equip $\A$ 
  with a fixed ideal $\m \subset \A$.  Most often we require that $\A$ be complete 
  with respect to $\m$, that is, we assume that the natural map 
  $\A \to \varprojlim \A/\m^n$ is an isomorphism.
  A typical setting is the power series ring $\K\fps{h}$ 
  over a field $\K$, equipped with its maximal ideal $\m = (h)$,
  or the ring of $p$-adic integers $\Z_{p} = \varprojlim \Zmod{p^n}$
  with its maximal ideal $(p)$.
\end{notation}

\begin{notation}[racks]
  A \emph{rack} or \emph{automorphic set} $(Q,\ast)$
  is a set $Q$ equipped with an operation $\ast \colon Q \times Q \to Q$
  such that every right translation $x \mapsto x \ast y$ 
  is an automorphism of $(Q,\ast)$.  This is equivalent to saying
  that the $\A$-linear map $c_Q \colon \A{Q} \tensor \A{Q} \to \A{Q} \tensor \A{Q}$
  defined by $c_Q \colon x \tensor y \mapsto y \tensor (x \ast y)$ for all $x,y \in Q$
  is a Yang-Baxter operator over the ring $\A$.
  
  Two rack elements $y,z \in Q$ are called \emph{behaviourally equivalent},
  denoted $y \equiv z$, if they satisfy $x \ast y = x \ast z$ for all $x \in Q$.
  This is equivalent to saying that $y,z$ have the same image
  under the inner representation $\rho \colon Q \to \Inn(Q)$.
  As usual, a matrix $f \colon Q^n \times Q^n \to \A$
  is called \emph{diagonal} if $f\smallindices{x_1,\dots,x_n}{y_1,\dots,y_n}$
  vanishes whenever $x_i \ne y_i$ for some index $i=1,\dots,n$.
  It is called \emph{quasi-diagonal} if $f\smallindices{x_1,\dots,x_n}{y_1,\dots,y_n}$
  vanishes whenever $x_i \not\equiv y_i$ for some index $i=1,\dots,n$.
\end{notation}

Quasi-diagonal maps play a crucial r\^ole in the classification
of deformations:

\begin{theorem} \label{thm:QuasiDiagonalDeformation}
  If the ring $\A$ is complete with respect to the ideal $\m$, 
  then every Yang-Baxter deformation $c$ of $c_Q$ over $(\A,\m)$
  is equivalent to a quasi-diagonal deformation.
  More explicitly this means that $c$ is conjugated over $(\A,\m)$ 
  to a deformation of the form $c_Q \circ (\id^{\tensor 2} + f)$ where 
  the deformation term $f \colon \A{Q^2} \to \m{Q^2}$ is quasi-diagonal.
\end{theorem}

% \begin{theorem} \label{thm:QuasiDiagonalDeformation}
%   Suppose that the ring $\A$ is complete with respect to the ideal $\m$.
%   Then every Yang-Baxter deformation $c$ of $c_Q$ over $(\A,\m)$ % $\A$ with respect to $\m$
%   is equivalent to a quasi-diagonal deformation, that is, $c$ is conjugated over $(\A,\m)$ 
%   to a deformation of the form $c_Q \circ (\id^{\tensor 2} + f)$ 
%   where the deformation term $f \colon \A{Q^2} \to \m{Q^2}$ is quasi-diagonal.
% \end{theorem}

% \begin{theorem} \label{thm:QuasiDiagonalDeformation}
%   Suppose that the ring $\A$ is complete with respect to the ideal $\m$.
%   Then every Yang-Baxter deformation of $c_Q$ over $(\A,\m)$ 
%   is equivalent to a quasi-diagonal deformation.
% \end{theorem}

% More explicitly this means that every Yang-Baxter deformation $c$ of $c_Q$ over $(\A,\m)$ 
% is conjugated to a deformation of the form $c_Q \circ (\id^{\tensor 2} + f)$ 
% where the deformation term $f \colon \A{Q^2} \to \m{Q^2}$ is quasi-diagonal.

There are thus two extreme cases in the deformation theory of racks:

\begin{enumerate}

\item \label{item:ExtremeDeformability}
  In the one extreme the rack $Q$ is trivial,
  whence $\rho \colon Q \to \Inn(Q)$ is trivial 
  and all elements of $Q$ are behaviourally equivalent.
  This is the initial setting in the theory of quantum invariants 
  and we cannot add anything new here.

\item \label{item:ExtremeRigidity}
  In the other extreme, where $\rho \colon Q \to \Inn(Q)$ is injective,
  % the Yang-Baxter cohomology of $c_Q$ collapses to rack cohomology,
  all elements of $Q$ are behaviourally distinct,
  and every deformation of $c_Q$ is equivalent to a diagonal deformation.
  This is the setting of rack cohomology and colouring invariants. 

\end{enumerate}

This result confirms a plausible observation:
the more innner symmetries $Q$ has, the less deformations $c_Q$ admits.
Our theorem makes the transition between the two extremes precise and 
quantifies the degree of deformability of set-theoretic Yang-Baxter operators.

\begin{example}
  Consider a group $(G,\cdot)$ that is generated by 
  one of its conjugacy classes $Q \subset G$.
  Then $(Q,\ast)$ is a rack with respect 
  to conjugation $x \ast y = y^{-1} \cdot x \cdot y$,
  and we have a natural isomorphism $\Inn(Q,\ast) \cong \Inn(G,\cdot)$.
  % \[
  % \begin{CD}
  %   (G,\cdot) @>{\rho}>> \Inn(G,\cdot) \\
  %   @A{\mathrm{inc}}AA @AA{\cong}A \\
  %   (Q,\ast) @>{\rho}>> \Inn(Q,\ast)
  % \end{CD}
  % \]
  If the centre of $G$ is trivial, then 
  $\rho \colon G \isoto \Inn(G)$ is a group isomorphism.
  For the operator $c_Q$ % associated to the rack $Q$ 
  the injectivity of $\rho \colon Q \to \Inn(Q)$ implies that 
  % (according to case \ref{item:ExtremeRigidity} above)
  every Yang-Baxter deformation of $c_Q$ 
  is equivalent to a diagonal deformation.
  If, moreover, the order $|G|$ is finite and invertible 
  in $\A$, then $c_Q$ is rigid over $\A$.
\end{example}

As pointed out above, diagonal deformations have received much attention
over the last 20 years \cite{FreydYetter:1989,FennRourke:1992,Yetter:2001,CarterEtAl:2003}.
It is reassuring that Theorem \ref{thm:QuasiDiagonalDeformation} justifies
this short-cut in the case where $\rho \colon Q \to \Inn(Q)$ is injective.
In general, however, the simplified ansatz of diagonal deformations 
may miss some interesting Yang-Baxter deformations, namely
those that are quasi-diagonal but not diagonal.
For more detailed examples and applications see \sref{sec:Examples}.

\subsection{Yang-Baxter cohomology}

Our approach to proving Theorem \ref{thm:QuasiDiagonalDeformation}
follows the classical paradigm of studying algebraic deformation theory 
via cohomology, as expounded by Gerstenhaber \cite{Gerstenhaber:1964}.
Since it may be of independent interest, we state here our main 
cohomological result, which in degree $2$ proves the infinitesimal 
version of Theorem \ref{thm:QuasiDiagonalDeformation}.

In the previous article \cite{Eisermann:2005} 
I introduced Yang-Baxter cohomology $H_\YB^*(c_Q;\m)$,
which encodes infinitesimal deformations of $c_Q$
over a ring $\A$ with respect to the ideal $\m$ 
(\sref{sub:YangBaxterDeformations}).
There I calculated the second cohomology $H_\YB^2(c_Q;\m)$ under 
the hypothesis that the order of inner automorphism group 
$\Inn(Q)$ is finite and invertible in the ring $\A$.
The main application was to finite racks $Q$
deformed over the ring $\A = \Q\fps{h}$.
In many cases the results of \cite{Eisermann:2005} 
imply that $c_Q$ is rigid over $\Q\fps{h}$.

In the present article we calculate Yang-Baxter cohomology $H_\YB^*(c_Q;\m)$ 
in the modular case, which had previously been left in suspense
\cite[Question 39]{Eisermann:2005}.
As our main result we establish the following classification; 
for detailed definitions and proofs we refer to \sref{sec:HomotopyRetraction}.

\begin{theorem} \label{thm:QuasiDiagonalCohomology}
  The quasi-diagonal subcomplex $C_\Delta^*(c_Q;\m) \subset C_\YB^*(c_Q;\m)$ is a homotopy retract, 
  whence the induced map $H^*_\Delta(c_Q;\m) \to H_\YB^*(c_Q;\m)$ is an isomorphism.
  % Suppose that $\A$ is a commutative ring with unit and fix an ideal $\m \subset \A$.
  % Let $Q$ be a rack and let $c_Q \colon \A{Q^2} \to \A{Q^2}$, $x \tensor y \mapsto y \tensor x^y$, 
  % be the associated Yang-Baxter operator over $\A$.
  % Then the quasi-diagonal subcomplex $C_\Delta^*(c_Q;\m) \subset C_\YB^*(c_Q;\m)$
  % is a homotopy retract.  The inclusion thus induces a natural isomorphism 
  % $H^*_\Delta(c_Q;\m) \isoto H_\YB^*(c_Q;\m)$ in cohomology.
\end{theorem}

Notice that, contrary to \cite{Eisermann:2005}, 
we no longer require the rack $Q$ to be finite,
nor do we impose any restrictions on the characteristic of 
the ring $\A$.

\begin{remark}
  Yang-Baxter cohomology includes rack cohomology $H_\Rack^*(Q;\Lambda)$
  as its diagonal part, as explained in \sref{sec:Diagonal},
  where $\Lambda$ is a module over some ring $\K$.
  % Rack cohomology has been intensively studied in recent years.
  If $|\Inn(Q)|$ is invertible in $\K$, then $H_\Rack^*(Q;\Lambda)$ 
  is trivial in a certain sense \cite{EtingofGrana:2003}.
  In the modular case, however, it leads to non-trivial 
  and topologically interesting Yang-Baxter deformations
  (see Example \ref{exm:A5} below).
  % In the modular case, where the order of $\Inn(Q)$ is not 
  % invertible in $\A$, it leads to non-trivial 
  % and topologically interesting Yang-Baxter deformations.
  It follows that Yang-Baxter cohomology of racks must
  be at least as complicated, and the modular case 
  stood out as a difficult yet promising problem.
  
  Theorem \ref{thm:QuasiDiagonalCohomology} solves this problem:
  it shows that the right object to study is the quasi-diagonal 
  subcomplex $C_\Delta^*$, situated between the strictly diagonal 
  complex $C_\Diag^*$ and the full Yang-Baxter complex $C_\YB^*$, 
  i.e., we have $C_\Diag^* \subset C_\Delta^* \subset C_\YB^*$.
  % $C_\Diag^*(c_Q;\m) \subset C_\Delta^*(c_Q;\m) \subset C_\YB^*(c_Q;\m)$ 
  We will see that the inclusion $C_\Diag^* \subset C_\YB^*$ 
  allows for a retraction $C_\YB^* \onto C_\Diag^*$, which entails
  that $H_\Diag^*$ is a direct summand of $H_\YB^*$.  In general, however,
  this is not a homotopy retraction and $H_\Diag^* \subsetneqq H_\YB^*$.
  The inclusion $\iota \colon  C_\Delta^* \subset C_\YB^*$ also
  allows for a retraction $\pi \colon C_\YB^* \onto C_\Delta^*$,
  such that $\pi \circ \iota = \id_\Delta$, and our main result 
  is the construction of a homotopy $\iota \circ \pi \simeq \id_\YB$.
\end{remark}
  
% \[
% \begin{diagram}
%   \node{} 
%   \node{} 
%   \node{C_\YB^*}
%   \\
%   \node{C_\Diag^*}
%   \arrow{ene,t}{\text{retracts}} 
%   \arrow{ese,b}{\text{retracts}} 
%   \node{} 
%   \node{} 
%   \\
%   \node{} 
%   \node{} 
%   \node{C_\Delta^*} 
%   \arrow[2]{n,r}{\text{homotopy retracts}}
% \end{diagram}
% \qquad\qquad
% \begin{diagram}
%   \node{} 
%   \node{} 
%   \node{C_\YB^*}
%   \arrow{wsw,t}{\text{retract}} 
%   \arrow[2]{s,r}{\text{homotopy retract}}
%   \\
%   \node{C_\Diag^*} 
%   \node{} 
%   \node{} 
%   \\
%   \node{} 
%   \node{} 
%   \node{C_\Delta^*} \arrow{wnw,b}{\text{retract}} 
% \end{diagram}
% \]

\begin{remark}
  Again we have two extreme cases that are particularly clear-cut:
  \begin{enumerate}
  \item 
    In the one extreme, if $Q$ is trivial, then all elements of $Q$ 
    are behaviourally equivalent.  In this case we trivially have $C_\Delta = C_\YB$.
  \item
    If $\rho \colon Q \to \Inn(Q)$ is injective,
    then all elements of $Q$ are behaviourally distinct.
    In this case quasi-diagonal means diagonal, whence $C_\Delta = C_\Diag$.
    % In the extreme case where $\rho \colon Q \to \Inn(Q)$ is injective,
    % quasi-diagonal means diagonal, and Yang-Baxter cohomology
    % collapses to rack cohomology (\sref{sec:Diagonal}).  
  \end{enumerate}

  In general $C_\Delta^*$ lies strictly between $C_\Diag^*$ and $C_\YB^*$.
  Even if $C_\Delta$ collapses to $C_\Diag$, this restrictive situation 
  is in general not rigid, and interesting deformations do arise 
  in the modular case (see Example \ref{exm:A5}).

  Depending on the structure of the rack $Q$, 
  the quasi-diagonal subcomplex $C_\Delta^*(c_Q;\m)$
  % $C_\Delta^*(c_Q;\m) \subset C_\YB^*(c_Q;\m)$ 
  can still be quite big, but in any case collapsing 
  the full Yang-Baxter complex $C_\YB^*(c_Q;\m)$
  to its quasi-diagonal subcomplex greatly simplifies the problem.
  In practical terms it reduces the complexity 
  from $|Q|^4$ unknowns to the order of $|Q|^2$ unknowns, 
  which in many cases makes it amenable to computer calculations.
  See \sref{sec:Examples} for illustrating examples.
\end{remark}

\subsection{How this article is organized}

Section \ref{sec:Definitions} recollects the relevant definitions 
concerning Yang-Baxter operators, their deformations, 
and the corresponding cohomology theory. 
It also gives explicit formulae in the case 
of racks, which is our main focus here.
Section \ref{sec:Diagonal} identifies diagonal
deformation with rack cohomology. 
Section \ref{sec:QuasiDiagonal} clarifies the relationship 
between diagonal deformations and rack cohomology,
and introduces quasi-diagonal deformations.
Section \ref{sec:HomotopyRetraction} proves our main result
in the infinitesimal case, by constructing a homotopy retraction
from the Yang-Baxter complex to its quasi-diagonal subcomplex.
Section \ref{sec:CompleteDeformations} extends the infinitesimal
result to complete deformations, and Section \ref{sec:Examples}
provides some examples and applications.
We conclude with some open question
in Section \ref{sec:OpenQuestions}.

%%%%%%%%%%%%%%%%%%%%%%%%%%%%%%%%%%%%%%%%%%%%%%%%%%%%%%%%%%%%%%%%%%%%%%%%%%%%%

\section{Yang-Baxter operators, deformations, and cohomology} \label{sec:Definitions}

\subsection{Yang-Baxter operators} \label{sub:YangBaxterOperators}

% Throughout this article $\A$ denotes a commutative ring with unit.
% All modules will be $\A$-modules, all maps between modules will be 
% $\A$-linear, and all tensor products will be formed over $\A$.
% For every $\A$-module $V$ we denote by $V^{\tensor n}$ 
% the $n$-fold tensor product $V \tensor \cdots \tensor V$ of $n$ copies of $V$.

\begin{definition}
  A \emph{Yang-Baxter operator} % on an $\A$-module $V$ 
  is an automorphism $c\colon V \tensor V \to V \tensor V$ satisfying
  the \emph{Yang-Baxter equation}, also called \emph{braid relation},
  \begin{equation} \label{eq:YangBaxterEquation}
    (c \tensor \id_V)(\id_V \tensor c)(c \tensor \id_V) = 
    (\id_V \tensor c)(c \tensor \id_V)(\id_V \tensor c) 
    \qquad\text{in}\quad \Aut_\A(V^{\tensor3}).
  \end{equation}
\end{definition}

This equation first appeared in theoretical physics (Yang \cite{Yang:1967}, 
Baxter \cite{Baxter:1972,Baxter:1982}, Faddeev \cite{Faddeev:1984}).  
It also has a very natural interpretation in terms of Artin's braid group $B_n$ 
\cite{Artin:1947,Birman:1974} and its tensor representations:  
the automorphisms $c_1,\dots,c_{n-1} \colon V^{\tensor n} \to V^{\tensor n}$ defined by 
% $c_i = \id_{\smash{V}}^{\tensor(i{-}1)} \tensor c \tensor \id_{\smash{V}}^{\tensor(n{-}i{-}1)}$
% satisfy the braid relations $c_i c_j c_i = c_j c_i c_j$ for $|i-j| = 1$
% and $c_i c_j = c_j c_i$ for $|i-j| \ge 2$.
\[
c_i = \id_{\smash{V}}^{\tensor(i{-}1)} 
\;\tensor\; c \;\tensor\; \id_{\smash{V}}^{\tensor(n{-}i{-}1)} ,
\qquad\text{ or in graphical notation }
c_i = \picc{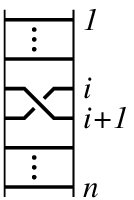} , % \quad c_i^{-1} = \picc{braid-generator} ,
\]
satisfy the well-known braid relations
\begin{xalignat*}{2}
  c_i c_j c_i & = c_j c_i c_j \quad\text{if $|i-j| = 1$,} 
  &
  \picc{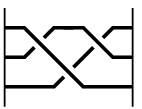} & = \picc{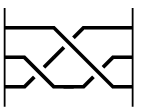},
  \\
  c_i c_j & = c_j c_i \quad\text{if $|i-j| \ge 2$,}
  & 
  \picc{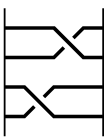} & = \picc{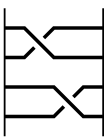}.
\end{xalignat*}

% \begin{xalignat*}{2}
%   c_i c_j c_i & = c_j c_i c_j \quad\text{if $|i-j| = 1$,} 
%   & c_i c_j & = c_j c_i \quad\text{if $|i-j| \ge 2$.}
%   \\
%   \picc{braid-relation3a} & = \picc{braid-relation3b} &
%   \picc{braid-relation4a} & = \picc{braid-relation4b}
% \end{xalignat*}

\begin{remark}
  A graphical notation for tensor calculus was first used 
  by Penrose \cite{Penrose:1971};  for a brief discussion 
  of its history see \cite{JoyalStreet:1991}.
  This notation has the obvious advantage to appeal to our geometric vision.
  More importantly, it incorporates a profound relationship with knot theory,
  and its rigorous formulation in terms of tensor categories directly 
  leads to knot invariants. % \cite{Turaev:1988,Turaev:1994,Kassel:1995}.
  More explicitly, Yang-Baxter operators induce invariants 
  of knots and links in $\mathbb{S}^3$ as follows,
  see Turaev \cite[chap.\,I]{Turaev:1994} 
  or Kassel \cite[chap.\,X]{Kassel:1995}.
  \[
  \begin{CD}
    B_n @>\text{Yang-Baxter}>\text{representation}> \Aut(V^{\tensor n}) \\
    @V{\text{closure}}VV @VV{\text{trace}}V \\
    \{\text{links}\} @>>{\text{invariant}}> \A
  \end{CD}
  \]

  Each link $L$ can be presented as the closure of 
  some braid. % according to the Theorem of Alexander-Markov.
  This braid acts on $V^{\tensor n}$ as defined above, and
  a suitably deformed trace maps it to the ring $\A$.
  In favourable cases the result does not depend on the choice 
  of braid, and thus defines an invariant of the link $L$.
\end{remark}

\subsection{Quandles and racks} \label{sub:QuandlesAndRacks}

In every group $(G,\cdot)$ the conjugation $a \ast b = b^{-1} \cdot a \cdot b$
enjoys the following properties:
\begin{enumerate}
\item[(Q1)]
  For every element $a$ we have $a \ast a = a$. 
  \hfill (idempotency)
\item[(Q2)]
  Every right translation $\varrho(b) \colon a \mapsto a \ast b$ is a bijection.
  \hfill (right invertibility)
\item[(Q3)]
  For all $a,b,c$ we have $(a \ast b) \ast c = (a \ast c) \ast (b \ast c)$.
  \hfill (self-distributivity)
\end{enumerate}

Taking these properties as axioms, Joyce \cite{Joyce:1982} defined 
a \emph{quandle} to be a set $Q$ equipped with a binary operation
$\ast \colon Q \times Q \to Q$ satisfying (Q1--Q3). 
% Quandles thus encode the algebraic properties of conjugation;
% this axiomatic approach is most natural for studying situations 
% where group multiplication is absent or of a secondary nature
Independently, Matveev \cite{Matveev:1982} studied 
the equivalent notion of \emph{distributive groupoid}.
Following Brieskorn \cite{Brieskorn:1988}, an \emph{automorphic set} 
is only required to satisfy (Q2--Q3): these two axioms are equivalent 
to saying that every right translation is an automorphism of $(Q,\ast)$. 
The shorter term \emph{rack} was suggested % preferred 
by Fenn and Rourke \cite{FennRourke:1992},
going back to the terminology used by J.H.\,Conway 
in correspondence with G.C.\,Wraith in 1959.
Such structures appear naturally in the study of braid actions 
\cite{Brieskorn:1988} and provide set-theoretic solutions of 
the Yang-Baxter equation \cite{Drinfeld:1990}:

\begin{proposition}
  Given a set $Q$ with binary operation $\ast \colon Q \times Q \to Q$,
  we can consider the free module $V = \A{Q}$ with basis $Q$ over $\A$
  and define the operator 
  \[
  c_Q \colon \A{Q}\tensor\A{Q} \to \A{Q}\tensor\A{Q} \quad\text{by}\quad
  a \tensor b \mapsto b \tensor (a \ast b) \quad\text{for all}\quad a,b \in Q.
  \]
  Then $c_Q$ is a Yang-Baxter operator if and only if $Q$ is a rack.
  \qed
\end{proposition}

Throughout this article we will use the following notation. % conventions.
A \emph{homomorphism} between two racks $(Q,\ast)$ and $(Q',\ast')$
is a map $\phi \colon Q \to Q'$ satisfying 
$\phi(a \ast b) = \phi(a) \ast' \phi(b)$ for all $a,b\in Q$.
Racks and their homomorphisms form a category.

The \emph{automorphism group} $\Aut(Q)$ % of the rack $Q$
consists of all bijective homomorphisms $\phi \colon Q \to Q$.  
We adopt the convention that automorphisms of $Q$ act on the right, 
written $a^\phi$, which means that their composition $\phi\psi$ 
is defined by $a^{(\phi\psi)} = (a^\phi)^\psi$ for all $a \in Q$.

Each $a \in Q$ defines an automorphism $\rho(a) \in \Aut(Q)$ defined by $x \mapsto x \ast a$.
For every $\phi \in \Aut(Q)$ we have $\rho(a^\phi) = \rho(a)^\phi$. 
% In particular we obtain $\rho(a \ast b) = \rho(a) \ast \rho(b)$ for all $a,b \in Q$.
The group $\Inn(Q)$ of \emph{inner automorphisms} is the normal subgroup 
of $\Aut(Q)$ generated by all right translations $\rho(a)$, where $a \in Q$.
The map $\rho \colon Q \to \Inn(Q)$ is the \emph{inner representation} of $Q$.
It satisfies $\rho(a \ast b) = \rho(a) \ast \rho(b)$,
that is, it maps the operation of the rack $Q$
to conjugation in the group $\Inn(Q)$.

% The group $\Inn(Q)$ of \emph{inner automorphisms} is the subgroup 
% of $\Aut(Q)$ generated by all right translations $\rho(a)$, where $a \in Q$. 
% We thus obtain a map $\rho \colon Q \to \Inn(Q)$ 
% called the \emph{inner representation} of $Q$.
% It satisfies $\rho(a \ast b) = \rho(a) \ast \rho(b)$,
% that is, it maps the operation of the rack $Q$
% to conjugation in the group $\Inn(Q)$.

\begin{notation}
  In view of the representation $\rho \colon Q \to \Inn(Q)$, 
  we often write $a^b$ for the operation $a^{\rho(b)} = a \ast b$.
  Conversely, it will sometimes be convenient to write $a \ast b$ 
  for the conjugation $a^b = b^{-1} a b$ in a group.
  In neither case will there be any danger of confusion.
\end{notation}

\begin{definition}
  Two elements $x,y \in Q$ are \emph{behaviourally equivalent}
  if $a \ast x = a \ast y$ for all $a \in Q$.
  This means that $\rho(x) = \rho(y)$,
  and will be denoted by $x \equiv y$ for short.
\end{definition}

% \begin{definition}
%   Two elements $x,y \in Q$ are \emph{behaviourally equivalent},
%   denoted $x \equiv y$, if $\rho(x) = \rho(y)$.  More explictly 
%   this means that $a \ast x = a \ast y$ for all $a \in Q$.
% \end{definition}

\subsection{Deformations and Yang-Baxter cohomology} \label{sub:YangBaxterDeformations}

We are interested here in set-theoretic solutions of the Yang-Baxter equation
and their deformations within the space of Yang-Baxter operators over some ring.
A typical setting is the power series ring $\K\fps{h}$ over a field $\K$, 
% (Example \ref{exm:PowerSeriesRing})
equipped with its maximal ideal $\m = (h)$.
In positive characteristic it is equally natural to consider
the ring of $p$-adic integers $\Z_{p} = \varprojlim \Zmod{p^n}$ 
% (Example \ref{exm:PadicIntegers}) 
with its maximal ideal $(p)$.

\begin{definition}
  We fix an ideal $\m$ in the ring $\A$.  Consider an $\A$-module $V$ 
  and a Yang-Baxter operator $c \colon V \tensor V \to V \tensor V$.
  % Suppose that $c \colon V \tensor V \to V \tensor V$ is a Yang-Baxter operator.

  A map $\tilde{c} \colon V \tensor V \to V \tensor V$ is called 
  a \emph{Yang-Baxter deformation} of $c$ with respect to $\m$
  if $\tilde{c}$ is itself a Yang-Baxter operator and 
  satisfies $\tilde{c} \equiv c$ modulo $\m$.

  An \emph{equivalence transformation}, 
  or \emph{gauge equivalence} with respect to $\m$,
  is an automorphism $\alpha \colon V \to V$ 
  satisfying $\alpha \equiv \id_V$ modulo $\m$.

  Two Yang-Baxter operators $c$ and $\tilde c$ 
  are called \emph{equivalent} % (with respect to $\m$) 
  if there exists an equivalence transformation $\alpha \colon V \to V$ such that 
  $\tilde c = (\alpha\tensor\alpha)^{-1} \circ c \circ (\alpha\tensor\alpha)$.
\end{definition}

In order to study deformations it is useful, as usual, to linearize 
the problem by considering infinitesimal deformations, where $\m^2=0$.
To this end we recall the definition of Yang-Baxter cohomology
$H_\YB(c;\m)$ that encodes infinitesimal deformations. 

\begin{definition}
  The Yang-Baxter cochain complex $C_\YB^*(c;\m)$ consists 
  of the $\A$-modules $C^n = \Hom( V^{\tensor n}, \m V^{\tensor n} )$. 
  For each $f \in C^n$ we define the partial coboundary $d^n_i f \in C^{n+1}$ by
  \begin{equation}
    d^n_i f = 
    \left(c_n \cdots c_{i+1}\right)^{-1} (f \tensor \id_V) \left(c_n \cdots c_{i+1}\right)
    - \left(c_1 \cdots c_i\right)^{-1} (\id_V \tensor f) \left(c_1 \cdots c_i\right) .
  \end{equation}
  
  This formula becomes more suggestive in graphical notation: % \cite{Eisermann:2005}:
  \begin{equation} \label{eq:CoboundaryGraphic}
    d^n_i f =  \quad + \quad \picc{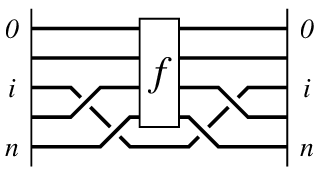} \quad - \quad \picc{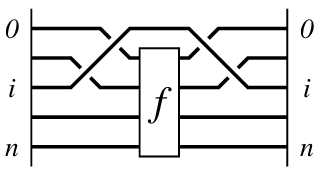} .
  \end{equation}
  The coboundary $d^n \colon C^n \to C^{n+1}$ 
  is defined as the alternating sum $d^n = \sum_{i=0}^{n} \, (-1)^i d^n_i$.
  % This a cochain complex, denoted by $C_\YB^*(c;\m) = (C^*,d^*)$. % for emphasis.
  % Its cohomology $H_\YB^*(c;\m)$ is called the Yang-Baxter cohomology 
  % of the operator $c$ with respect to the ideal $\m$.
\end{definition}

\begin{proposition} \label{prop:YangBaxterCoboundary}
  We have $d^{n+1}_i \circ d^n_j = d^{n+1}_{j+1} \circ d^n_i$ 
  for $i \le j$, whence $d^{n+1} \circ d^n = 0$.
\end{proposition}

\begin{proof}
  The hypothesis that $c$ be a Yang-Baxter operator ensures 
  that $d^{n+1}_i d^n_j = d^{n+1}_{j+1} d^n_i$ for $i \le j$.
  This can be proven by a straightforward computation, and 
  is most easily verified using the graphical calculus suggested above.
  It follows, as usual, that all terms cancel each other in pairs 
  to yield $d^{n+1} \circ d^n = 0$.
\end{proof}

\begin{definition}
  The cochain complex $C_\YB^*(c;\m) = (C^*,d^*)$ is called 
  the \emph{Yang-Baxter cochain complex}, and its cohomology $H_\YB^*(c;\m)$ 
  is called the \emph{Yang-Baxter cohomology}
  of the operator $c$ with respect to the ideal $\m$.
\end{definition}

% \begin{remark}
%   Instead of an ideal $\m$ in a ring $\A$ one can also define 
%   Yang-Baxter cohomology for any module $\m$ over a ring $\K$.  
%   This is, in fact, a special case of our definition:
%   we can form the $\K$-algebra $\A = \K \oplus \m$ by setting 
%   $u v = 0$ for all $u,v \in \m$, that is, we equip $\A$ 
%   with the product $(a,u) \cdot (b,v) = (ab, av+bu)$.
%   The $\K$-algebra homomorphism $\varepsilon \colon \A \to \K$, 
%   defined by $\varepsilon(1)=1$ and $\varepsilon(u)=0$ 
%   for all $u \in \m$, has precisely the ideal $\m$ as its kernel.
%   By construction we have $\m^2 = 0$ and $\A$ is complete with respect to $\m$.
%   For $\K = \m = \Zmod{n}$, for example, we thus obtain $\A = \Zmod{n}[h]/(h^2)$.
%   The ring $\A = \Zmod{n^2}$ with $\m = (n)$ and $\K = \A/\m = \Zmod{n}$,
%   on the other hand, shows that our definition is much more comprehensive.
% \end{remark}  

\begin{proposition}
  The second cohomology $H_\YB^2(c;\m)$ classifies infinitesimal 
  Yang-Baxter deformations: assuming $\m^2 = 0$, the deformation 
  $\tilde c = c \circ (\id_{\smash{V}}^{\tensor 2}+f)$ satisfies
  \begin{multline*}
    (\id_V \tensor \tilde c)^{-1}(\tilde c \tensor \id_V)^{-1}(\id_V \tensor \tilde c)^{-1} 
    (\tilde c \tensor \id_V)(\id_V \tensor \tilde c)(\tilde c \tensor \id_V) \\
    = 
    (\id_V \tensor c)^{-1}(c \tensor \id_V)^{-1}(\id_V \tensor c)^{-1} 
    (c \tensor \id_V)(\id_V \tensor c)(c \tensor \id_V) + d^2{f} .
  \end{multline*}
  This means that $\tilde c$ is a Yang-Baxter operator if and only if $d^2 f = 0$.
  Likewise, $c$ and $\tilde c$ are equivalent via conjugation 
  by $\alpha = (\id_V + g)$ if and only if $f = d^1 g$, because
  \[
  (\alpha\tensor\alpha)^{-1} \circ c \circ (\alpha\tensor\alpha) 
  = c \circ (\id_V^{\tensor 2} + d^1{g} ) .
  \]
\end{proposition}

\begin{remark}
  In the quantum case, where $c = \tau$, 
  we obtain $d f = 0$ for all $f \in C_\YB^*$.
  In particular there are no infinitesimal obstructions to deforming $\tau$: 
  \emph{every} deformation of $\tau$ satisfies the Yang-Baxter equation 
  modulo $\m^2$, and only higher-order obstructions are of interest.
  This explains why Yang-Baxter cohomology is absent in the quantum case.

  Infinitesimal obstructions are important, however,
  if $c \ne \tau$, for example for an operator $c_Q$ coming from 
  a non-trivial rack $Q$, the main object of interest to us here.
  In extreme cases they even allow us to conclude that $c_Q$ is rigid.
\end{remark}

\begin{example}
  Yang-Baxter cohomology can in particular be applied to study 
  the deformations of the Yang-Baxter operator $c_Q$ associated with a rack $Q$.
  The canonical basis $Q$ of $V = \A{Q}$ allows us to identify 
  each $\A$-linear map $f \colon \A{Q^n} \to \A{Q^n}$ with its matrix
  $f \colon Q^n \times Q^n \to \A$, related by the definition
  \[
  f\colon \left( x_1 \tensor \cdots \tensor x_n \right) \mapsto
  \sum_{y_1,\dots,y_n} f\indices{x_1,\dots,x_n}{y_1,\dots,y_n} 
  \cdot \left( y_1 \tensor \cdots \tensor y_n \right) .
  \]
  Matrix entries are thus denoted by $f\smallindices{x_1,\dots,x_n}{y_1,\dots,y_n}$
  with indices $\smallindices{x_1,\dots,x_n}{y_1,\dots,y_n} \in Q^n\times Q^n$.
  If $Q$ is infinite, then we use the tacit convention 
  that for each basis element $(x_1,\dots,x_n) \in Q^n$
  the coefficient $f\smallindices{x_1,\dots,x_n}{y_1,\dots,y_n}$
  is non-zero only for a finite number of $(y_1,\dots,y_n) \in Q^n$.

  For example, the identity $\id \colon \A{Q} \to \A{Q}$
  will be identified with the following matrix $Q \times Q \to \A$,
  which is usually called the Kronecker delta function:
  \[
  \id\indices{x}{y} = \begin{cases} 
    1 & \text{if $x = y$,} \\ 
    0 & \text{if $x \ne y$.} 
  \end{cases}
  \]
  In this notation the coboundary can be rewritten more explicitly as follows:
  \begin{align}
    \label{eq:CoboundaryFormula}
    (d_i^n f)\indices{x_0,\dots,x_n}{y_0,\dots,y_n} = 
    + & f\indices{x_0^{\phantom{x_i}},\dots,x_{i-1}^{\phantom{x_i}},x_{i+1},\dots,x_n}%
    {y_0^{\phantom{y_i}},\dots,y_{i-1}^{\phantom{y_i}},y_{i+1},\dots,y_n} 
    \cdot \id\indices{x_i^{x_{i+1}\cdots x_n}}{y_i^{y_{i+1}\cdots y_n}} 
    \\ \notag
    - & f\indices{x_0^{x_i},\dots,x_{i-1}^{x_i},x_{i+1},\dots,x_n}% 
    {y_0^{y_i},\dots,y_{i-1}^{y_i},y_{i+1},\dots,y_n} \cdot \id\indices{x_i}{y_i}.
  \end{align}
  % Again the coboundary $d^n f \colon C^n \to C^{n+1}$ 
  % is given by $d^n f = \sum_{i=0}^{n} \, (-1)^i d^n_i f$.
\end{example}

\begin{remark}
  Instead of an ideal $\m$ in a ring $\A$ one can also define 
  the Yang-Baxter cochain complex $C_\YB^*(c;\m)$ and its cohomology
   $H_\YB^*(c;\m)$ for any module $\m$ over a ring $\K$.  
  Both points of view become equivalent in the infinitesimal setting:
  % Our definitions are motivated by the deformation theory 
  % over $\A$ with respect to some fixed ideal $\m$.
  % From an abstract viewpoint of cohomology theory,
  % the notation $H_\YB^*(c;\m)$ might be surprising at first sight:
  % one would rather expect a module $\m$ over some ring $\K$ 
  % than an ideal $\m$ in $\A$.  Notice, however, that both points of view 
  % become equivalent in the infinitesimal setting:
  \begin{itemize}
  \item
    First, if $\m^2 = 0$ in $\A$, then $\m$ is 
    a module over the quotient ring $\K = \A/\m$.
  \item
    Conversely, every $\K$-module $\m$ defines 
    a $\K$-algebra $\A = \K \oplus \m$ with $\m^2=0$.
  \end{itemize}

  Consider for example the power series ring $\K\fps{h}$ over a ring $\K$.
  In the infinitesimal setting we have $\A = \K\fps{h}/(h^2) = \K \oplus \m$ 
  where $\m = \K h$ is isomorphic with $\K$.

  The $p$-adic integers $\Z_{p} = \varprojlim \Zmod{p^n}$
  lead to the infinitesimal algebra $\A = \Z_{p}/(p^2) = \Zmod{p^2}$.
  Here $\m = (p) \cong \Zmod{p}$ and $\K = \A/\m = \Zmod{p}$,
  but the projection $\A \onto \K$ does not split.

  This shows that an infinitesimal algebra $\A$, 
  with $\m^2 = 0$, need not be of the form $\K \oplus \m$.
  This interesting phenomenon only appears in positive characteristic.
  It does not play any r\^ole for infinitesimal deformations, 
  but may become crucial for higher order deformations.
\end{remark}

\subsection{Yang-Baxter homology} \label{sub:YangBaxterHomology}

As could be expected, there is a homology theory dual to Yang-Baxter cohomology.
Even though we shall not explicitly use it in the sequel,
it may be illuminating to briefly sketch its construction.
Since we are interested in analyzing coefficient modules,
the standard approach would be to exploit the interplay between 
homology and cohomology via the Universal Coefficient Theorem,
see \cite{MacLane:1995}, \textsection III.4 and \textsection V.11.

\begin{remark}[traces]
  Let $\A$ be a ring and let $c \colon V \tensor V \to V \tensor V$ 
  be a Yang-Baxter operator.  
  We will assume that the $\A$-module $V$ is free of finite rank,
  so that we can define a trace $\tr \colon \End(V) \to \A$,
  see Lang \cite[\textsection XVI.5]{Lang:2002}.
  Slightly more general, it suffices to assume $V$ projective 
  and finitely generated over $\A$, see Turaev \cite[chap.\,1]{Turaev:1994}.
  Even though this hypothesis may seem restrictive, 
  it is precisely the setting of quantum knot invariants,
  where a trace $\tr \colon \End(V) \to \A$ is indispensable.
  Notice further that then $\End(V^{\tensor n}) = \End(V)^{\tensor n}$,
  and for each index $i=1,\dots,n$ we have a partial trace
  $\tr_i \colon \End(V)^{\tensor n} \to \End(V)^{\tensor(n-1)}$
  defined by contracting the $i$th tensor factor.
\end{remark}

% See Lang \cite[\textsection XVI.5]{Lang:2002}
% for the necessary yoga on tensor products.
% Notice also that as a consequence its dual $V^* = \Hom(V,\A)$ 
% is again free of finite rank, its bidual $V^*^*$ comes 
% equipped with a natural isomorphism $V \isoto V^*^*$,
% and we have a trace $\tr \colon \End(V) \to \A$.

\begin{definition}
  Given a Yang-Baxter operator $c \colon V \tensor V \to V \tensor V$,
  the Yang-Baxter chain complex $C^\YB_*(c)$ 
  consists of the $\A$-modules $C_n = \End(V^{\tensor n})$.
  We define the partial boundary $\partial_n^i \colon C_{n} \to C_{n-1}$ by
  \begin{equation}
    \partial_n^i f = 
    \tr_n \left[ \left( c_{n-1} \cdots c_i \right) 
      \circ f \circ \left( c_{n-1} \cdots c_i \right)^{-1} \right]
    - \tr_1 \left[ \left( c_1 \cdots c_{i-1} \right) 
      \circ f \circ \left( c_1 \cdots c_{i-1} \right)^{-1} \right].
  \end{equation}

  Again this formula becomes more suggestive in graphical notation,
  which is, as \eqref{eq:CoboundaryGraphic}, reminiscent of rope skipping:
  \begin{equation} \label{eq:BoundaryGraphic}
    \partial_n^i f = \quad + \quad \picc{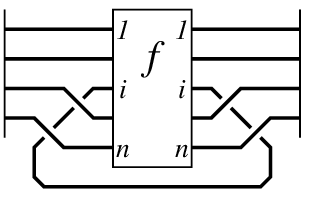} \quad - \quad \picc{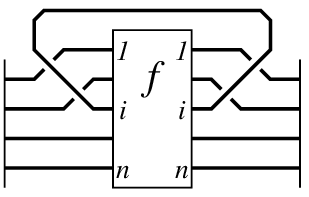} .
  \end{equation}

  The boundary $\partial_n \colon C_{n} \to C_{n-1}$ is defined 
  as the alternating sum $\partial_n = \sum_{i=1}^{n} \, (-1)^{i-1} \partial_n^i$.
\end{definition}

\begin{proposition}
  We have $\partial_{n-1}^j \circ \partial_n^i 
  = \partial_{n-1}^i \circ \partial_n^{j+1}$ for $i \le j$,
  whence $\partial_{n-1} \circ \partial_n = 0$.
\end{proposition}

% \begin{proposition}
%   The boundary satisfies $\partial_{n-1} \circ \partial_n = 0$,
%   so that $C^\YB_*(c) = (C_*,\partial_*)$ is indeed a chain complex.
%   Its homology $H^\YB_*(c)$ is called the \emph{Yang-Baxter homology} of $c$.
% \end{proposition}

\begin{proof}
  The hypothesis that $c$ be a Yang-Baxter operator ensures that 
  $\partial_{n-1}^j \circ \partial_n^i = \partial_{n-1}^i \circ \partial_n^{j+1}$ for $i \le j$.
  This can be proven by a straightforward computation, and
  is most easily verified using the graphical calculus suggested above.  
  It follows, as usual, that all terms cancel each other in pairs 
  to yield $\partial_{n-1} \circ \partial_n = 0$. 
\end{proof}

\begin{definition}
  The chain complex $C^\YB_*(c) = (C_*,\partial_*)$ is called
  the \emph{Yang-Baxter chain complex}, and its homology $H^\YB_*(c)$ 
  is called the \emph{Yang-Baxter homology} of the operator $c$.
\end{definition}

\begin{proposition}
  The dual complex $\Hom(C^\YB_*,\m)$ is naturally isomorphic to 
  the Yang-Baxter cochain complex $C_\YB^*(c;\m)$ defined above.
\end{proposition}

\begin{proof}
  The natural duality is induced by the duality pairing
  $\End(V^{\tensor n}) \tensor \End(V^{\tensor n}) \to \A$
  defined by $\langle f \mid g \rangle = \tr( f g )$.  
  In graphical notation this reads as 
  \[
  \langle f \mid g \rangle % = \tr( f g ) 
  = \picc[0.8]{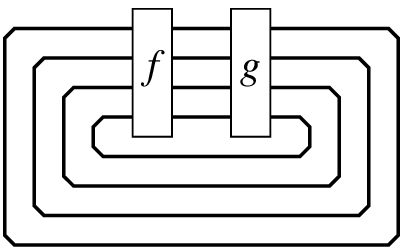} .
  \]

  The advantage of this notation is that all calculations 
  become self-evident.  In particular, we see that the coboundary 
  operator $d^*$ of Equation \eqref{eq:CoboundaryGraphic} is the dual 
  of the boundary operator $\partial_*$ of Equation \eqref{eq:BoundaryGraphic}:
  for $f \in C^\YB_{n+1}$ and $g \in C_\YB^n$ and all $i=1,\dots,n+1$ we have 
  \[
  \langle \partial_{n+1}^i f \mid g \rangle = \langle f \mid d^n_{i-1} g \rangle .
  \]
  In graphical notation this can be seen as follows:
  \begin{align*}
    \picc[0.8]{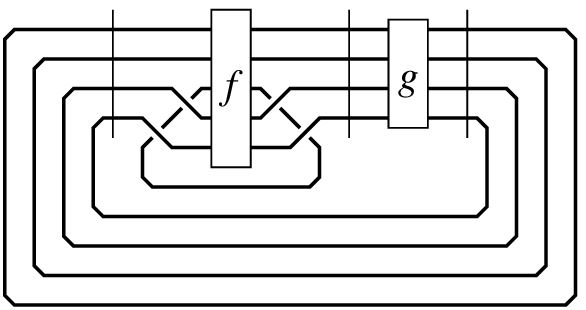} & = \picc[0.8]{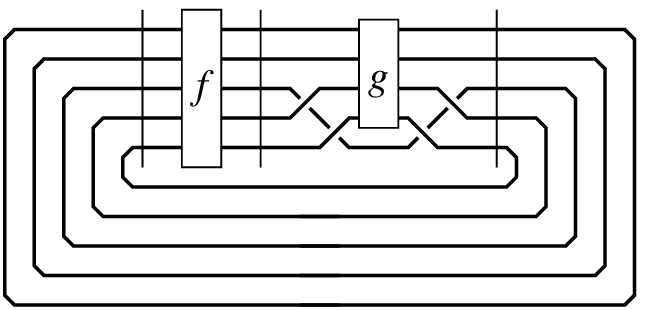} , \\
    \picc[0.8]{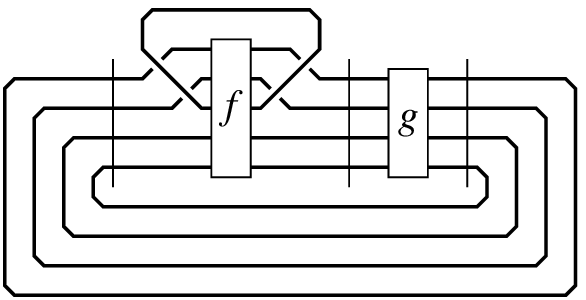} & = \picc[0.8]{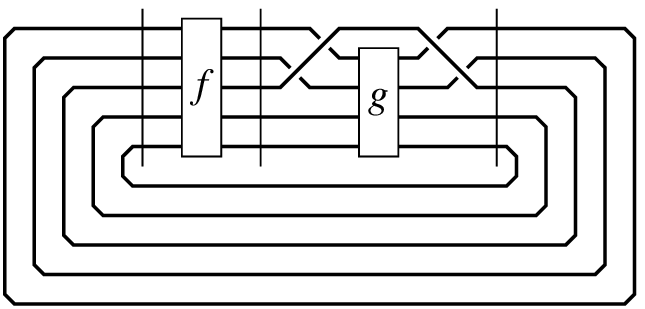} .
  \end{align*}
  We conclude that $\langle \partial_{n+1} f \mid g \rangle = \langle f \mid d^n g \rangle$
  as claimed.
  % \[
  % \langle \partial_{n+1}^i f \mid g \rangle 
  % = \langle f \mid d^n_{i-1} g \rangle 
  % \]
  % for all $n \in \N$ and $i=1,\dots,n+1$, and thus 
  % $\langle \partial_{n+1} f \mid g \rangle = \langle f \mid d^n g \rangle$.
\end{proof}

% The natural duality can be seen as follows:
% \begin{align*}
%   \Hom( C^\YB_n , \m ) &
%   \cong \Hom( V^{\tensor n} \tensor (V^*)^{\tensor n} , \m )
%   \cong \Hom( V^{\tensor n} , \Hom( (V^*)^{\tensor n} , \m ) ) \\ & 
%   \cong \Hom( V^{\tensor n} , \m \tensor V^{\tensor n} )
%   \cong \Hom( V^{\tensor n} , \m V^{\tensor n} )
%   = C_\YB^n
% \end{align*}
% The first isomorphism follows from $\End(V) \cong V \tensor V^*$.
% The second isomorphism is the universal property of the tensor product.
%   % $\Hom(A \tensor B, C) = \Hom( A, \Hom(B,C) )$.
% For the third isomorphism we use the evaluation map 
% $\m \tensor V^{\tensor n} \to \Hom( (V^*)^{\tensor n} , \m )$:
% given $a \tensor v \in \m \tensor V^{\tensor n}$, we send
% each form $v^* \in (V^*)^{\tensor n}$ to $a \cdot v^*(v) \in \m$.
% Under our hypothesis this is an isomorphism.
% The fourth isomorphism follows from the fact that the natural
% map $\m \tensor V^{\tensor n} \to \m V^{\tensor n}$ is an isomorphism,
% because $V^{\tensor n}$ is flat over $\A$, see \cite[\textsection V.8]{MacLane:1995}.
%   % or \cite[\textsection XVI.3]{Lang:2002}.

\begin{remark}
  In the case of a finite rack $Q$ and its associated Yang-Baxter operator $c_Q$,
  the chain complex $C^\YB_*$ can be described as follows.
  Starting from the canonical basis $Q$ of $V = \A{Q}$, 
  we obtain the basis $Q^n$ of $V^{\tensor n}$ and 
  then a basis $Q^n \times Q^n$ of $\End(V^{\tensor n})$.
  In analogy with our previous notation we denote by 
  $\smallpindices{x_1,\dots,x_n}{y_1,\dots,y_n}$ the endomorphism 
  that maps $x_1 \tensor\cdots\tensor x_n$ to $y_1 \tensor\cdots\tensor y_n$, 
  while mapping all other elements of the basis $Q^n$ to zero.
  % Since $V = \A{Q}$ has $Q$ as canonical basis,
  % we obtain a canonical basis $Q^*$ for $V^*$,
  % and a canonical isomorphism $V \cong V^*$ identifying $Q$ and $Q^*$.
  % The chain complex $C_n = \End(V^{\tensor n}) \cong V^{\tensor n} \tensor (V^*)^{\tensor n}$
  % has a canonical basis $Q^n \times Q^n$, whose elements will be denoted by
  % $\smallpindices{x_1,\dots,x_n}{y_1,\dots,y_n}$, in analogy with our previous notation.
  The boundary operator can then be rewritten more explicitly as follows:
  \begin{align}
    \label{eq:BoundaryFormula}
    \partial_n \pindices{x_1,\dots,x_n}{y_1,\dots,y_n} = \sum_{i=1}^n (-1)^{i-1} 
    \biggl[ \; & \pindices{x_1^{\phantom{x_i}},\dots,x_{i-1}^{\phantom{x_i}},x_{i+1},\dots,x_n}%
    {y_1^{\phantom{y_i}},\dots,y_{i-1}^{\phantom{y_i}},y_{i+1},\dots,y_n} 
    \cdot \id\pindices{x_i^{x_{i+1}\cdots x_n}}{y_i^{y_{i+1}\cdots y_n}} 
    \\ \notag
    - & \pindices{x_1^{x_i},\dots,x_{i-1}^{x_i},x_{i+1},\dots,x_n}% 
    {y_1^{y_i},\dots,y_{i-1}^{y_i},y_{i+1},\dots,y_n} 
    \cdot \id\pindices{x_i}{y_i} \; \biggr].
  \end{align}
  
  We see that the boundary formula \eqref{eq:BoundaryFormula}
  is dual to the coboundary formula \eqref{eq:CoboundaryFormula},
  which nicely illustrates the preceding proposition.
  Again diagonal chains form a subcomplex, which corresponds to 
  rack homology defined in \cite{FennRourke:1992,CarterEtAl:2001}.  
  The diagonal subcomplex $C^\Diag_* \subset C^\YB_*$ is a retract, 
  so that rack homology is a direct summand of Yang-Baxter homology.
  Moreover, under the above duality, $C^\Diag_* \subset C^\YB_*$ 
  is dual to $C_\Diag^* \subset C_\YB^*$.
\end{remark}

\begin{remark}
  The duality exhibited above 
  is graphically appealing and theoretically satisfying:
  it is reassuring to have the standard homology-cohomology pairing.
  Notice, however, that we have to restrict to free modules 
  of finite rank, or finitely generated projective modules.
  Yang-Baxter cohomology alone can be defined over arbitrary Yang-Baxter
  modules $(V,c)$, not necessarily projective or finitely generated.
  From this abstract viewpoint Yang-Baxter cohomology 
  thus seems more natural than homology.  
\end{remark}

\subsection{Non-Functoriality} \label{sub:NonFunctoriality}

Yang-Baxter cohomology and homology suffer from a curious defect:
they are not functorial with respect to homomorphisms of Yang-Baxter operators:

\begin{definition}
  A \emph{homomorphism} % or \emph{transformation} 
  between Yang-Baxter operators $c \colon V \tensor V \to V \tensor V$ and
  $\bar{c} \colon \bar{V} \tensor \bar{V} \to \bar{V} \tensor \bar{V}$
  % $(V,c)$ and $(\bar{V},\bar{c})$ over the ring $\A$
  is an $\A$-linear map $\phi \colon V \to \bar{V}$ such that 
  $\bar{c} \circ (\phi\tensor\phi) = (\phi\tensor\phi) \circ c$.
  \[
  \begin{CD}
    V \tensor V  @>{c}>> V \tensor V \\
    @V{\phi\tensor\phi}VV @VV{\phi\tensor\phi}V \\
    \bar{V} \tensor \bar{V} @>{\bar{c}}>> \bar{V} \tensor \bar{V}
  \end{CD}
  \]
  This condition ensures that $\phi$ induces for each $n$ a homomorphism 
  $\phi^{\tensor n} \colon V^{\tensor n} \to \bar{V}^{\tensor n}$ 
  that is equivariant with respect to the natural action 
  of Artin's braid group $B_n$.
  % between the modules $V^{\tensor n}$ and $\bar{V}^{\tensor n}$
\end{definition}

\begin{example}
  A map $\phi \colon Q \to \bar{Q}$ is a homomorphism
  between two racks $Q$ and $\bar{Q}$ if and only if only if its $\A$-linear
  extension $\phi \colon \A{Q} \to \A\bar{Q}$ is a homomorphism
  between the associated Yang-Baxter operators $c_{Q}$ and $c_{\bar{Q}}$.
\end{example}

\begin{remark}
  Given a homomorphism $\phi$ between Yang-Baxter operators $c$ and $\bar{c}$, 
  we would expect a natural cochain homomorphism 
  $\phi^* \colon C_\YB^*(\bar{c};\m) \to C_\YB^*(c;\m)$ as well as
  a natural chain homomorphism $\phi_* \colon C^\YB_*(c) \to C^\YB_*(\bar{c})$.
  The definitions of $C^\YB_n(c) = \End(V^{\tensor n})$ and 
  $C_\YB^n(c;\m) = \Hom( V^{\tensor n}, \m V^{\tensor n} )$, 
  however, do not lend themselves to any obvious construction.
  This difficulty persists even if $V$ is free of finite rank.  
  In degree $2$ the problem is that in a general deformation of 
  $c \colon V \tensor V \to V \tensor V$ both factors interact non-trivially.
  This does not respect the product structure of $\phi^{\tensor n}$.

  To be more explicit, consider a homomorphism $\phi \colon Q \to \bar{Q}$ between 
  finite racks. % and its $\A$-linear extension $\phi \colon \A{Q} \to \A\bar{Q}$.
  We can define a map $\phi^* \colon C_\YB^*(c_{\bar{Q}};\m) \to C_\YB^*(c_Q;\m)$
  by setting
  \begin{equation}
    \label{eq:YBInducedMap}
    (\phi^* f)\indices{x_1,\dots,x_n}{y_1,\dots,y_n} 
    = f\indices{\phi(x_1),\dots,\phi(x_n)}{\phi(y_1),\dots,\phi(y_n)} .
  \end{equation}
  Even though this is the natural candidate,
  it does in general not define a cochain map, that is, 
  we usually have $\phi^* \circ d_{\smash{\bar{Q}}}^* \ne d_{\smash{Q}}^* \circ \phi^*$.
\end{remark}

\begin{example}
  Consider a rack $Q$.  
  The inner automorphism group $\Inn(Q)$ acts naturally on $Q$.
  The set of orbits $\bar{Q} = Q/\Inn(Q)$ can be regarded 
  as a trivial rack, in which case the quotient map 
  $\phi \colon Q \to \bar{Q}$ becomes a rack homomorphism.  

  % Let $Q$ be a non-trivial rack and consider 
  % the unique rack homomorphism $\phi \colon Q \to \bar{Q}$ 
  % to the trivial rack $\bar{Q} = \{*\}$.  In this case every 
  % Every cochain $f \in C_\YB^n(c_{\bar{Q}};\m)$ is just a constant. 
  % Its coboundary vanishes, and so $\phi^* d_{\smash{\bar{Q}}}^* f = 0$.

  Consider a cochain $f \in C_\YB^n(c_{\bar{Q}};\m)$.
  The coboundary $d_{\smash{\bar{Q}}}^*$ vanishes, 
  so that $\phi^* d_{\smash{\bar{Q}}}^* f = 0$.
  In general, however, we have $d_Q^* \phi^* f \ne 0$.
  To see this consider $y,z \in Q$ satisfying $y \not\equiv z$,
  which means that there exists $x \in Q$ such that $x^y \ne x^z$.
  % A small calculation shows that 
  % $( d_Q^1 \phi^* f )\smallindices{x,y}{x,z} 
  % = - ( \phi^* f )\smallindices{y}{z}
  % = - f\smallindices{\smash{\phi(y)}}{\smash{\phi(z)}}$. 
  We find 
  \begin{align*}
    \Bigr( d_Q^1 (\phi^* f) \Bigr)\indices{x,y}{x,z}
    & = (\phi^* f)\indices{y}{z} \cdot \id\indices{x^y}{x^z} 
    - (\phi^* f)\indices{y}{z} \cdot \id\indices{x}{x} \\
    & - (\phi^* f)\indices{x}{x} \cdot \id\indices{y}{z} 
    + (\phi^* f)\indices{x^y}{x^z} \cdot \id\indices{y}{z}
    = - f\indices{\phi(y)}{\phi(z)} .
  \end{align*}
  This is in general not zero, whence $d_Q^* \phi^* f \ne \phi^* d_{\smash{\bar{Q}}}^* f$.
  (For an explicit example take $\bar{Q}$ to be trivial, so that
  $d_{\smash{\bar{Q}}}^*=0$ and $f$ can be chosen arbitrarily.)
  We conclude that the natural candidate 
  $\phi^* \colon C_\YB^*(c_{\bar{Q}};\m) \to C_\YB^*(c_Q;\m)$
  is not a cochain map.
\end{example}

%%%%%%%%%%%%%%%%%%%%%%%%%%%%%%%%%%%%%%%%%%%%%%%%%%%%%%%%%%%%%%%%%%%%%%%%%%%%%

\section{Diagonal deformations} \label{sec:Diagonal}

% Even though it is not strictly necessary for our main theorems,
% we should like to clarify the relationship with other (co)homology theories:
% first of all rack cohomology (\sref{sub:RackCohomology}), and secondly 
% the dual notion of homology (\sref{sub:YangBaxterHomology}).
% We also discuss the lack of functoriality (\sref{sub:NonFunctoriality}).

% \subsection{Diagonal deformations and rack cohomology} \label{sub:RackCohomology}

In \sref{sec:Definitions} % the preceding discussion 
we have considered arbitrary deformations of $c_Q$.  
The problem becomes much easier if we concentrate on deformations of the form 
$\tilde{c}(a \tensor b) = \lambda(a,b) \cdot c_Q(a \tensor b)$,
where $\lambda \colon Q \times Q \to \Lambda$ 
is a map into some abelian group $\Lambda$.  
% This approach has the advantage to greatly simplify calculations.
Such deformations are classified by rack cohomology:

\begin{definition}
  Let $Q$ be a rack and let $\Lambda$ be an abelian group (written additively).
  We consider the cochain complex $C_\Rack^n = C_\Rack^n(Q;\Lambda)$ formed 
  by all maps $\lambda \colon Q^n \to \Lambda$.
  The coboundary $\delta^n \colon C_\Rack^n \to C_\Rack^{n+1}$ is defined by 
  \begin{align} \label{eq:RackCohomology}
    (\delta^n \lambda)(a_0,\dots,a_n) = \sum_{i=1}^n (-1)^i 
    \Bigl[ \; &\lambda(a_0^{\phantom{a_i}}, \dots, a_{i-1}^{\phantom{a_i}}, a_{i+1}, \dots, a_n ) 
    \\[-2mm] \notag
    - & \lambda(a_0^{a_i}, \dots, a_{i-1}^{a_i}, a_{i+1}, \dots, a_n ) \; \Bigr].
  \end{align}
  This defines a cochain complex $(C_\Rack^\ast,\delta^\ast)$,
  whose cohomology $H_\Rack^n(Q;\Lambda)$ is called the
  \emph{rack cohomology} of $Q$ with coefficients in $\Lambda$.
\end{definition}

\begin{remark}
  It is easily seen that $\tilde{c}$ is a Yang-Baxter operator
  if and only if $\lambda$ is a rack cocycle, see Gra\~na \cite{Grana:2002}.
  Likewise, $\tilde{c}$ and $c_Q$ are equivalent if and only if
  $\lambda$ is a coboundary.
\end{remark}

\begin{remark}
  As in group theory, the second cohomology group $H_\Rack^2(Q;\Lambda)$
  is in bijective correspondence with equivalence classes of 
  central extension $\Lambda \curvearrowright \tilde{Q} \to Q$,
  see \cite{Eisermann:2003,Eisermann:2005}.
\end{remark}

Let us make explicit how rack cohomology fits into 
the more general framework of Yang-Baxter cohomology.  
% First of all, there are several choices for the ground ring $\A$.
% For rack cohomology with coefficients in an abelian group $\Lambda$
% we can work with the group algebra $\A = \K\Lambda$ over a ring $\K$,
% or some other algebra $\A$ in which $\Lambda \subset \A^\times$
% is a subgroup of invertible elements.
Suppose that $\Lambda$ is a module over some ring $\K$.
We can form the $\K$-algebra $\A = \K \oplus \Lambda$ 
by setting $u v = 0$ for all $u,v \in \Lambda$,
that is, we equip $\A$ with the product $(a,u) \cdot (b,v) = (ab, av+bu)$.
For $\K = \Lambda = \Zmod{n}$, for example, we obtain $\A = \K[h]/(h^2)$.
% For the sake of concreteness one could think of $\K = \Lambda = \Zmod{n}$, 
% say, in which case we obtain $\A = \K[h]/(h^2)$.

We have an augmentation homomorphism $\varepsilon \colon \A \to \K$ defined 
by $\varepsilon(1)=1$ and $\varepsilon(u)=0$ for all $u \in \Lambda$.
The augmentation ideal $\m = \ker(\varepsilon)$ thus coincides with $\Lambda$.
Notice also that the additive group $\Lambda$ is isomorphic 
to the multiplicative subgroup $1 + \m$ of the ring $\A$.

If we consider diagonal deformations 
\[
\tilde{c}(a\tensor b) = \bigl(1+\lambda(a,b)\bigr) \cdot c_Q(a \tensor b)
\quad\text{with $\lambda(a,b) \in \m$,}
\]
then we see that rack cohomology 
naturally embeds into Yang-Baxter cohomology: 

\begin{proposition} \label{prop:RackYBCohomology}
  The rack cochain complex $C_\Rack^*(Q;\Lambda)$ is naturally isomorphic
  to the diagonal subcomplex $C_\Diag^*$ of the Yang-Baxter cohomology $C_\YB^*(c_Q;\m)$.
\end{proposition}

\begin{proof}
  As usual, a matrix $f \colon Q^n \times Q^n \to \m$
  is called \emph{diagonal} if $f\smallindices{x_1,\dots,x_n}{y_1,\dots,y_n}$
  vanishes whenever $x_i \ne y_i$ for some index $i=1,\dots,n$.
  Diagonal cochains form a subcomplex of $(C_\YB^*,d^*)$.

  Each diagonal matrix $f \colon Q^n \times Q^n \to \m$ 
  can be identified with the corresponding map $\lambda \colon Q^n \to \m$
  such that $f\smallindices{x_1,\dots,x_n}{x_1,\dots,x_n} = \lambda(x_1,\dots,x_n)$.
  Under this identification, the Yang-Baxter coboundary \eqref{eq:CoboundaryFormula}
  reduces to the rack coboundary \eqref{eq:RackCohomology}.
\end{proof}

% \subsection{Functoriality}

\begin{remark}
  Unlike the full Yang-Baxter complex $(C_\YB^*(c_Q,\m),d^*)$, 
  the diagonal complex $(C_\Diag^*(c_Q,\m),d^*)$ and 
  rack cohomology \eqref{eq:RackCohomology} \emph{are} functorial in $Q$.
\end{remark}

% \subsection{Retraction}

\begin{proposition}
  There exists a retraction $r \colon C_\YB^* \to C_\Diag^*$ 
  of cochain complexes, whence rack cohomology $H_\Rack^*(Q;\Lambda)$ is
  a direct summand of Yang-Baxter cohomology $H_\YB^*(c_Q;\m)$.
\end{proposition}

\begin{proof}
  The obvious idea turns out to work.  % For $f \in C_\YB^n$ we set 
  We define $r^n \colon C_\YB^n \to C_\Diag^n$ by
  \[
  (r^n f)\smallindices{x_1,\dots,x_n}{y_1,\dots,y_n}
  := \begin{cases} 
    f\smallindices{x_1,\dots,x_n}{y_1,\dots,y_n}
    & \text{ if $x_i = y_i$ for all $i=1,\dots,n$,} 
    \\
    0 & \text{ otherwise.}
  \end{cases}
  \]
  The coboundary formula \eqref{eq:CoboundaryFormula}
  shows that $d_i^n \circ r^n = r^{n+1} \circ d_i^n$,
  whence $d \circ r = r \circ d$.
  By construction we have $r(C_\YB^*) = C_\Diag^*$ 
  and $r|C_\Diag^* = \id$, so that $r$ is a retraction, as desired.
\end{proof}

\begin{remark}
  The example of a trivial rack $Q$ shows that Yang-Baxter $H_\YB^*(c_Q;\m)$
  is in general much bigger than rack cohomology $H_\Rack^*(Q;\Lambda)$,
  so we cannot capture all information by diagonal deformations alone.
  In order to do so, we have to consider the more general notion 
  of \emph{quasi-diagonal} deformations, as explained below.
\end{remark}

% \subsection{Relation with cobraided bialgebras}

% FRT-construction
% cobraided bialgebras
% cohomology of cobraided bialgebras

% Yetter's book
% Gerstenhaber-Schack
% other references

\section{Quasi-diagonal deformations} \label{sec:QuasiDiagonal}

% \subsection{Quasi-diagonal deformations} \label{sub:QuasiDiagonal}

A matrix $f \colon Q^n \times Q^n \to \A$ is called 
\emph{quasi-diagonal} if $f\smallindices{x_1,\dots,x_n}{y_1,\dots,y_n}$
vanishes whenever $x_i \not\equiv y_i$ for some index $i=1,\dots,n$.

\begin{proposition}
  The quasi-diagonal cochains of the Yang-Baxter complex $(C_\YB^*,d^*)$
  form a subcomplex, denoted $(C_\Delta^*,d^*)$.
\end{proposition}

\begin{remark}
  Restricted to the subcomplex $C_\Delta^*$ of quasi-diagonal cochains, 
  the coboundary $d^n \colon C_\Delta^n \to C_\Delta^{n+1}$ 
  takes the form $d^n f = \sum_{i=1}^{n} \, (-1)^i d^n_i f$ with
  \[
  (d^n_i f)\indices{x_0,\dots,x_n}{y_0,\dots,y_n} = 
  \biggl( f\indices{x_0^{\phantom{x_i}},\dots,x_{i-1}^{\phantom{x_i}},x_{i+1},\dots,x_n}%
  {y_0^{\phantom{y_i}},\dots,y_{i-1}^{\phantom{y_i}},y_{i+1},\dots,y_n} 
  - f\indices{x_0^{x_i},\dots,x_{i-1}^{x_i},x_{i+1},\dots,x_n}% 
  {y_0^{y_i},\dots,y_{i-1}^{y_i},y_{i+1},\dots,y_n} \biggr) \cdot \id\indices{x_i}{y_i}.
  \]
  % Restricted to the subcomplex $C_\Delta^*$ of quasi-diagonal cochains, 
  % the coboundary $d^n \colon C_\Delta^n \to C_\Delta^{n+1}$ takes the form   
  % \begin{align}
  %   \label{eq:QuasidiagonalCoboundary}
  %   (d^n f)\indices{x_0,\dots,x_n}{y_0,\dots,y_n} = \sum_{i=1}^n (-1)^i 
  %   \biggl( \; & f\indices{x_0^{\phantom{x_i}},\dots,x_{i-1}^{\phantom{x_i}},x_{i+1},\dots,x_n}%
  %   {y_0^{\phantom{y_i}},\dots,y_{i-1}^{\phantom{y_i}},y_{i+1},\dots,y_n} 
  %   \\ \notag
  %   - & f\indices{x_0^{x_i},\dots,x_{i-1}^{x_i},x_{i+1},\dots,x_n}% 
  %   {y_0^{y_i},\dots,y_{i-1}^{y_i},y_{i+1},\dots,y_n} \; \biggr) \cdot \id\indices{x_i}{y_i}.
  % \end{align}
  
  This illustrates, in explicit terms, that the quasi-diagonal subcomplex
  is half-way between Yang-Baxter cohomology \eqref{eq:CoboundaryFormula} 
  and rack cohomology \eqref{eq:RackCohomology}.

  We should point out that the quasi-diagonal subcomplex $C_\Delta^*$ coincides 
  with the Yang-Baxter cochain complex $C_\YB^*$ if the rack $Q$ is trivial.
  % in this case $Q \to \Inn(Q)$ maps all elements to the identity, 
  % and so all elements of $Q$ are behaviourally equivalent.
  On the other hand, $C_\Delta^*$ coincides with the rack cochain complex $C_\Rack^*$ 
  if the inner representation $\rho \colon Q \to \Inn(Q)$ is injective:
  in this case $x \equiv y$ means $x = y$, and quasi-diagonal means diagonal.
\end{remark}

% \subsection{Non-Functoriality}

\begin{remark}
  Like the full Yang-Baxter complex $C_\YB^*(c_Q,\m)$, 
  the quasi-diagonal complex $C_\Delta^*(c_Q,\m)$ 
  is \emph{not} functorial in the rack $Q$.
  Every rack homomorphism $\phi \colon Q \to \bar{Q}$ induces a map
  $\phi_\Delta^* \colon C_\Delta(c_{\bar{Q}},\m) \to C_\Delta(c_Q,\m)$
  defined by 
  \begin{equation}
    \label{eq:QuasiDiagInducedMap}
    (\phi_\Delta^* f)\indices{x_1,\dots,x_n}{y_1,\dots,y_n} 
    = f\indices{\phi(x_1),\dots,\phi(x_n)}{\phi(y_1),\dots,\phi(y_n)} 
  \end{equation}
  for all % $x_1,\dots,x_n,y_1,\dots,y_n \in Q$ with 
  $x_1 \equiv y_1, \dots, x_n \equiv y_n$ in $Q$.
  This natural map, however, is in general not a cochain map.
  A concrete example can be constructed as follows.
\end{remark}

\begin{example}
  Consider a non-trivial rack $\bar{Q}$ 
  and choose $\bar{x},\bar{y} \in \bar{Q}$
  such that $\bar{x}^{\bar{y}} \ne \bar{x}$.
  Assume that $\phi \colon Q \to \bar{Q}$ is a rack homomorphism,
  $\phi(x) = \bar{x}$ and $\phi(y) = \phi(z) = \bar{y}$ with $y \ne z$.
  The easiest example is the trivial extension $Q = \bar{Q} \times \{1,2\}$,
  which also ensures that $y = (\bar{y},1)$ and 
  $z = (\bar{y},2)$ are behaviourally equivalent.
  For each cochain $f \in C^1(c_{\bar{Q}},\m)$ we find
  \begin{align}
    \label{eq:Naturality1}
    (d^1 \phi^* f)\indices{x,y}{x,z}
    & = \left( (\phi^* f)\indices{x^y}{x^z} -(\phi^* f)\indices{x}{x} \right) 
    \cdot \id\indices{y}{z} = 0
    \qquad\text{as opposed to}
    \\
    \label{eq:Naturality2}
    (\phi^* d^1 f)\indices{x,y}{x,z} 
    & = (d^1 f)\indices{\bar{x},\bar{y}}{\bar{x},\bar{y}} 
    = \left( f \indices{\bar{x}^{\bar{y}}}{\bar{x}^{\bar{y}}}
    - f \indices{\bar{x}}{\bar{x}} \right) \id\indices{\bar{y}}{\bar{z}}
    = f \indices{\bar{x}^{\bar{y}}}{\bar{x}^{\bar{y}}}
    - f \indices{\bar{x}}{\bar{x}} .
  \end{align}
  Since $\bar{x}^{\bar{y}} \ne \bar{x}$,
  the cochain $f$ can be so chosen that 
  the last difference is non-zero.
\end{example}

The difference between equations \eqref{eq:Naturality1} 
and \eqref{eq:Naturality2} disappears for equivariant cochains:

\begin{definition} \label{def:FullyEquivariant}
  A cochain $f \in C^n(c_Q,\m)$ is \emph{fully equivariant} if it satisfies 
  \[
  f\indices{x_1,\dots,x_n}{y_1,\dots,y_n} =
  f\indices{x_1^{g_1},\dots,x_n^{g_n}}{y_1^{g_1},\dots,y_n^{g_n}} 
  \]
  for all $x_1,\dots,x_n,y_1,\dots,y_n \in Q$ and $g_1,\dots,g_n \in \Inn(Q)$.
\end{definition}

\begin{definition} \label{def:Entropic}
  A cochain $f \in C^n(c_Q,\m)$ is called \emph{entropic} 
  if it is quasi-diagonal and fully equivariant.  
  Such cochains are characterized by the condition
  $d^n_0 f = \dots = d^n_n f = 0$, in other words, 
  all partial coboundaries vanish \cite[Lemma 30]{Eisermann:2005}.
  In particular, entropic cochains are cocycles; the submodule
  of entropic cocycles is denoted by $\Ent^*(c_Q,\m) \subset Z_\YB^*(c_Q,\m)$.
\end{definition}

\begin{remark}
  For every rack $Q$ we have $C_\YB^*(c_Q,\m) \supset 
  C_\Delta^*(c_Q,\m) \supset \Ent^*(c_Q,\m)$ by definition.
  Every rack homomorphism $\phi \colon Q \to \bar{Q}$ induces maps 
  % a map $\phi_\Delta^* \colon \Ent^*(c_{\bar{Q}},\m) \to \Ent^*(c_Q,\m)$.  
  \[
  \begin{CD}
    C_\YB^*(c_Q,\m) @<{\supset}<{}< C_\Delta^*(c_Q,\m) @<{\supset}<{}< \Ent^*(c_Q,\m) \\
    @A{\phi^*}AA @AA{\phi_\Delta^*}A @AA{\phi_\Ent^*}A \\
    C_\YB^*(c_{\bar{Q}},\m) @<{\supset}<{}< C_\Delta^*(c_{\bar{Q}},\m) @<{\supset}<{}< \Ent^*(c_{\bar{Q}},\m) .
  \end{CD}
  \]
  Here $\phi^*$ is defined by \eqref{eq:YBInducedMap}, 
  whereas $\phi_\Delta^*$ is defined by \eqref{eq:QuasiDiagInducedMap}, 
  and the map $\phi_\Ent^*$ is obtained from $\phi_\Delta^*$ by restriction.
  In general $\phi^*$ and $\phi_\Delta^*$ are \emph{not} cochain maps,
  as pointed out above.  Only the third map $\phi_\Ent^*$ is always 
  a cochain map because $C_\Ent^* \subset Z_\YB^*$ is a trivial subcomplex.
  %
  % For a trivial rack $\bar{Q}$ every cochain $f \in C^*(c_{\bar{Q}},\m)$
  % is entropic because $d^n_0 = \dots = d^n_n = 0$.
  % We thus have $C_\YB^*(c_{\bar{Q}},\m) = C_\Delta^*(c_{\bar{Q}},\m) = \Ent^*(c_{\bar{Q}},\m)$.
  % If $\phi \colon Q \to \bar{Q}$ is a rack homomorphism,
  % then the induced map
  % $\phi_\Delta^* \colon C_\Delta^*(c_{\bar{Q}},\m) \to C_\Delta^*(c_Q,\m)$  
  % is a cochain map.
\end{remark}

% \subsection{Retraction}

The main goal of this article is to show that 
$C_\Delta^* \subset C_\YB^*$ is a homotopy retract.
We point out that a much weaker statement
follows easily from the definition:

\begin{proposition} \label{prop:QuasiRectraction}
  There exists a retraction $r \colon C_\YB^* \to C_\Delta^*$, % of cochain complexes, 
  whence quasi-diagonal Yang-Baxter cohomology $H_\Delta^*(c_Q;\m)$ is
  a direct summand of Yang-Baxter cohomology $H_\YB^*(c_Q;\m)$.
\end{proposition}

\begin{proof}
  Again the obvious idea turns out to work.  
  We define $r^n \colon C_\YB^n \to C_\Delta^n$ by
  \[
  (r^n f)\smallindices{x_1,\dots,x_n}{y_1,\dots,y_n}
  := \begin{cases} 
    f\smallindices{x_1,\dots,x_n}{y_1,\dots,y_n}
    & \text{ if $x_i \equiv y_i$ for all $i=1,\dots,n$,} 
    \\
    0 & \text{ otherwise.}
  \end{cases}
  \]
  The coboundary formula \eqref{eq:CoboundaryFormula}
  shows that $d_i^n \circ r^n = r^{n+1} \circ d_i^n$,
  whence $d \circ r = r \circ d$.
  By construction we have $r(C_\YB^*) = C_\Delta^*$ 
  and $r|C_\Delta^* = \id$, so $r$ is a retraction, as desired.
\end{proof}

%%%%%%%%%%%%%%%%%%%%%%%%%%%%%%%%%%%%%%%%%%%%%%%%%%%%%%%%%%%%%%%%%%%%%%%%%%%%%

% \section{Homotopy retraction to quasi-diagonal deformations} 
\section{Constructing a homotopy retraction} \label{sec:HomotopyRetraction}

Having set the scene in the preceding sections, we can now study 
the subcomplex $C_\Delta^* \subset C_\YB^*$ of quasi-diagonal cochains.  
It is easy to see that it is a retract, but it is more delicate 
to prove that it is in fact a homotopy retract.  
To this end we introduce an auxiliary filtration 
$C_\YB^* \supset C_1^* \supset \dots \supset C_\infty^* = C_\Delta^*$
of subcomplexes (\sref{sub:Filtration}) and prove that each complex
homotopy retracts to its successor (\sref{sub:Homotopies}).
It then suffices to compose these piecewise homotopies in order to
obtain the desired homotopy retraction from $C_\YB^*$ to $C_\Delta^*$
(\sref{sub:CompositeHomotopy}).

\subsection{An intermediate filtration} \label{sub:Filtration}

We now turn to the problem of finding a homotopy retraction 
to the subcomplex $C_\Delta^* \subset C_\YB^*$ of quasi-diagonal cochains.
The construction of Proposition \ref{prop:QuasiRectraction} 
is nice and simple, but unfortunately the retraction 
$r \colon C_\YB^* \to C_\Delta^*$ does not seem to be 
a homotopy retraction, i.e.\ it is quite likely \emph{not} 
homotopic to the identity of $C_\YB^n$.  
% (See Remark \ref{rem:Equivariance} below.)

For reasons that will become clear in the following calculations,
it is rather complicated to directly define a homotopy retraction to 
the subcomplex $C_\Delta^* \subset C_\YB^*$. % of quasi-diagonal cochains.  
The approach that we present here resolves this difficulty 
by induction on a judiciously chosen filtration
\[
C_\YB^* = C_0^* \supset C_1^* \supset C_2^* \supset 
\dots \supset C_\infty^* = C_\Delta^* .
\]
This will allow us to construct the 
homotopy retraction $C_\YB^* \onto C_\Delta^*$
as the composition of partial retractions 
$p^*_m \colon C_m^* \onto C_{m+1}^*$
which are much easier to understand.
Figuratively speaking, we thus construct the deformation from 
$C_\YB^*$ to $C_\Delta^*$ by a piecewise linear path.

\begin{definition}
  For each $m \in \N$ we define $C_m^* \subset C_\YB^*$ to be the subcomplex
  of cochains that are quasi-diagonal in the last $m$ variables.  More explicitly:
  \[
  C_m^n := \bigl\{ f \in C_\YB^n \mid f\smallindices{x_1,\dots,x_n}{y_1,\dots,y_n} = 0
  \text{ if $x_i \not\equiv y_i$ for some index $i$ with $n-m < i \le n$} \bigr\} .
  \]
  In each degree $n$ we thus obtain a filtration 
  $C_\YB^n = C_0^n \supset C_1^n \supset \dots \supset C_n^n$
  that stabilizes at $C_n^n$: obviously $C_m^n = C_n^n$ for all $m > n$.  
  In each degree $n$ its limit is thus $\bigcap_{m} \, C_m^n = C_n^n$.
\end{definition}

\begin{lemma}
  The coboundary $d^n \colon C_\YB^n \to C_\YB^{n+1}$ 
  satisfies $d^n(C_m^n) \subset C_m^{n+1}$ for each $m \in \N$.
  In other words, $(C_m^*,d^*|_{C_m^*})$ is a subcomplex of $(C_\YB^*,d^*)$.
\end{lemma}

\begin{proof}
  Suppose that $f \in C_m^n$.  Formula \eqref{eq:CoboundaryFormula}
  for the partial coboundary shows that $d^n_i f$ is again in $C_m^{n+1}$.
  The same thus holds for the coboundary $d^n f = \sum_{i=0}^{n} \, (-1)^i d^n_i f$.
\end{proof}

\begin{notation}
  We will suppress the explicit mention of the coboundary map 
  and denote the complex $(C_\YB^*,d^*)$ simply by $C_\YB^*$.
  Likewise we write $C_m^*$ for the complex $(C_m^*,d^*|_{C_m^*})$.
\end{notation}

\subsection{Cochain homotopies} \label{sub:Homotopies}

We wish to show that the inclusion 
$\iota^*_{\smash{m+1}} \colon C_{\smash{m+1}}^* \into C_m^*$ 
is a \emph{homotopy retract}.  To this end we shall construct 
a cochain map $p^*_m \colon C_m^* \onto C_{\smash{m+1}}^*$ 
such that $p^*_m \circ \iota^*_{\smash{m+1}} = \id^*_{\smash{m+1}}$ 
and a cochain homotopy $\iota^*_{\smash{m+1}} \circ p^*_m 
\simeq \id^*_m \colon C_m^* \to C_m^*$.
Such a projection $p^*_m$ is called a \emph{homotopy retraction}
see \cite[\textsection II.2]{MacLane:1995}.
Recall that a \emph{cochain homotopy} is a map $s^n_m \colon C_m^n \to C_m^{n-1}$ 
such that $p^n_m - \id^n_m = d^{n-1} \circ s^n_m + s^{n+1}_m \circ d^n$.
In the sequel we will prefer the sign convention 
$d^{n-1} \circ s^n_m - s^{n+1}_m \circ d^n$, which is logically equivalent.
% We will show, moreover, that $p$ is a \emph{strong homotopy retract},
% in the sense that $p^{n-1}_m \circ s^n_m = 0$ and $s^n_m \circ \iota^n_{m+1} = 0$.

\begin{remark}
  We call the set $\Delta = \{ (x,y) \in Q^2 \mid x \equiv y \}$  the \emph{quasi-diagonal}.  
  On its complement $\Delta^c = \{ (x,y) \in Q^2 \mid x \not\equiv y \}$
  we choose a map $\psi \colon \Delta^c \to Q^2$, $(x,y) \mapsto (u,v)$
  such that $u \ne v$ but $u^x = v^y$.  It is easy to see that such a map exists:
  the inequivalence $x \not\equiv y$ means that the inner automorphisms 
  $z \mapsto z \ast x$ and $z \mapsto z \ast y$ are different.
  This is equivalent to saying that their inverses $z \mapsto z \tsa x$ and $z \mapsto z \tsa y$ 
  are different: there exists $z \in Q$ such that $u = z \tsa x$ differs 
  from $v = z \tsa y$.  In other words we have $u \ne v$ but $u^x = v^y$.
\end{remark}

\begin{definition}
  Fix $n,m \in \N$.  For $m \ge n$ we define 
  $s^n_m \colon C_m^n \to C_m^{n-1}$ to be the zero map.
  For $0 \le m \le n-1$ we set $k := n-m$
  and define $s^n_m \colon C_m^n \to C_m^{n-1}$ as follows:
  \[
  (s^n_m f)\smallindices{x_2,\dots,x_n}{y_2,\dots,y_n}
  := \begin{cases} 
    f\smallindices{x_2,\dots,x_{k-1},u,x_k,\dots,x_n}{y_2,\dots,y_{k-1},v,y_k,\dots,y_n}
    & \text{ if $x_k \not\equiv y_k$, with $(u,v) = \psi(x_k,y_k)$,} 
    \\
    0 & \text{ if $x_k \equiv y_k$.}
  \end{cases}
  \]
  This induces a map $t^n_m := d^{n-1} \circ s^n_m - s^{n+1}_m \circ d^n \colon C_m^n \to C_m^n$. 
\end{definition}

\begin{theorem}
  The cochain map $p^n_m := \id^n_m - (-1)^{n-m} t^n_m \colon C_m^n \to C_m^n$ 
  sends $C_m^n$ to the subcomplex $C_{m+1}^n$ and restricts to 
  the identity on $C_{m+1}^n$.  By construction, the maps $\id^*_m$ and $p^*_m$
  are homotopy equivalent, and thus $C_{m+1}^* \into C_m^*$ is a homotopy retract.
  The inclusion thus induces an isomorphism $H^*(C_{m+1}^*) \isoto H^*(C_m^*)$
  on cohomology.
\end{theorem}

\begin{proof}
  We first have to check that $p$ is a cochain map.  
  This follows at once from its definition:
  \begin{align*}
    d^n \circ p^n_m     & = d^n - (-1)^{n-m}\bigl[ d^n d^{n-1} s^n_m - d^n s^{n+1}_m d^n \bigr] ,
    \\
    p^{n+1}_m \circ d^n & = d^n + (-1)^{n-m}\bigl[ d^n s^{n+1}_m d^n - s^{n+2}_m d^{n+1} d^n \bigr] .
  \end{align*}

  The two properties $p^n_m(C_m^n) \subset C_{\smash{m+1}}^n$ 
  and $p^n_m|{C_{\smash{m+1}}^n} = \id^n_{\smash{m+1}}$ 
  will be established in the following two lemmas.
  The remaining statements are standard consequences 
  of cochain homotopy \cite[\textsection II.2]{MacLane:1995}.
\end{proof}

\begin{lemma}
  The map $t^n_m$ satisfies 
  $(t^n_m f)\smallindices{x_1,\dots,x_n}{y_1,\dots,y_n} =
  (-1)^k f\smallindices{x_1,\dots,x_n}{y_1,\dots,y_n}$
  whenever $x_k \not\equiv y_k$.
  % This means that $p^n_m := \id^n_m - (-1)^k t^n_m$ 
  % maps $C_m^n$ to the subcomplex $C_{m+1}^n$.
\end{lemma}

\begin{proof}
  We recall our short-hand notation $k := n - m$.
  We will calculate $t^n_m \colon C_m^n \to C_m^n$ by making the individual terms 
  $d^{n-1}_i \circ s^n_m$ and $s^{n+1}_m \circ d^n_i$ explicit for $i = 0,\dots,n$. 
  Let $f \in C_m^n$ and assume $x_k \not\equiv y_k$.
  We shall distinguish the three cases $i \le k-2$ and $i = k-1$ and $i \ge k$.

  \bigskip \textit{First case.} 
  For $i=0,\dots,k-2$ we find:
  \begin{align*}
    (d^{n-1}_i s^n_m f)\indices{x_1,\dots,x_n}{y_1,\dots,y_n} =
    + & (s^n_m f)\indices%
    {x_1^{\phantom{x_{i+1}}},\dots,x_{i}^{\phantom{x_{i+1}}},x_{i+2},\dots,x_k,\dots,x_n}%
    {y_1^{\phantom{y_{i+1}}},\dots,y_{i}^{\phantom{y_{i+1}}},y_{i+2},\dots,y_k,\dots,y_n} 
    \cdot \id\indices{x_{i+1}^{x_{i+2}\cdots x_n}}{y_{i+1}^{y_{i+2}\cdots y_n}} 
    \\ 
    - & (s^n_m f)\indices%
    {x_1^{x_{i+1}},\dots,x_{i}^{x_{i+1}},x_{i+2},\dots,x_k,\dots,x_n}% 
    {y_1^{y_{i+1}},\dots,y_{i}^{y_{i+1}},y_{i+2},\dots,y_k,\dots,y_n} 
    \cdot \id\indices{x_{i+1}}{y_{i+1}} 
    \\ 
    = + & f\indices%
    {x_1^{\phantom{x_{i+1}}},\dots,x_{i}^{\phantom{x_{i+1}}},x_{i+2},\dots,u,x_k,\dots,x_n}%
    {y_1^{\phantom{y_{i+1}}},\dots,y_{i}^{\phantom{y_{i+1}}},y_{i+2},\dots,v,y_k,\dots,y_n} 
    \cdot \id\indices{x_{i+1}^{x_{i+2}\cdots x_n}}{y_{i+1}^{y_{i+2}\cdots y_n}} 
    \\ 
    - & f\indices%
    {x_1^{x_{i+1}},\dots,x_{i}^{x_{i+1}},x_{i+2},\dots,u,x_k,\dots,x_n}% 
    {y_1^{y_{i+1}},\dots,y_{i}^{y_{i+1}},y_{i+2},\dots,v,y_k,\dots,y_n} 
    \cdot \id\indices{x_{i+1}}{y_{i+1}} 
    \\ 
    = \phantom{+} & (d^n_i f)\indices%
    {x_1,\dots,x_{k-1},u,x_k,\dots,x_n}%
    {y_1,\dots,y_{k-1},v,y_k,\dots,y_n}
    \\ 
    = \phantom{+} & (s^{n+1}_m d^n_i f)\indices{x_1,\dots,x_n}{y_1,\dots,y_n} . 
    %-----------------------------------------------------------------------
    \intertext{The third of these four equalities needs justification.
      We have to verify that 
      \begin{align*}
        x_{i+1}^{x_{i+2} \cdots x_{k-1} x_k \cdots x_n} 
        & = y_{i+1}^{y_{i+2} \cdots y_{-1} y_k \cdots y_n}
        \intertext{is equivalent to}
        x_{i+1}^{x_{i+2} \cdots x_{k-1} u x_k \cdots x_n} 
        & = y_{i+1}^{y_{i+2} \cdots y_{k-1} v y_k \cdots y_n}.
      \end{align*}
      We can assume that $x_j \equiv y_j$ for all $k < j \le n$, 
      otherwise the factors involving $f$ vanish by our hypothesis $f \in C_m^n$.  
      So we only have to show that
      \begin{align*}
        x_{i+1}^{x_{i+2} \cdots x_{k-1} x_k} 
        & = y_{i+1}^{y_{i+2} \cdots y_{-1} y_k}
        \intertext{is equivalent to}
        x_{i+1}^{x_{i+2} \cdots x_{k-1} u x_k} 
        & = y_{i+1}^{y_{i+2} \cdots y_{k-1} v y_k}.
      \end{align*}
      This follows from $(a \ast u) \ast x_k = (a \ast x_k) \ast (u \ast x_k)$
      and $(b \ast v) \ast y_k = (b \ast y_k) \ast (v \ast y_k)$,
      and our construction $(u,v) = \psi(x_k,y_k)$ 
      ensures that $u \ast x_k = v \ast y_k$.
      %
      % This calculation shows a little more for $f \in C_m^n$:
      % if $x_j \not\equiv y_j$ for some $j$ with $k < j \le n$, then 
      % $(d^{n-1}_i s^n_m f)\smallindices{x_1,\dots,x_n}{y_1,\dots,y_n} 
      % = (s^{n+1}_m d^n_i f)\smallindices{x_1,\dots,x_n}{y_1,\dots,y_n} = 0$.
      \bigskip \newline 
      \textit{Second case.}
      For $i=k-1$ we find:}
    %-----------------------------------------------------------------------
    (d^{n-1}_{k-1} s^n_m f)\indices{x_1,\dots,x_n}{y_1,\dots,y_n} =
    + & (s^n_m f)\indices%
    {x_1^{\phantom{x_{k}}},\dots,x_{k-1}^{\phantom{x_{k}}},x_{k+1},\dots,x_n}%
    {y_1^{\phantom{y_{k}}},\dots,y_{k-1}^{\phantom{y_{k}}},y_{k+1},\dots,y_n} 
    \cdot \id\indices{x_{k}^{x_{k+1}\cdots x_n}}{y_{k}^{y_{k+1}\cdots y_n}} 
    \\ 
    - & (s^n_m f)\indices%
    {x_1^{x_{k}},\dots,x_{k-1}^{x_{k}},x_{k+1},\dots,x_n}% 
    {y_1^{y_{k}},\dots,y_{k-1}^{y_{k}},y_{k+1},\dots,y_n} 
    \cdot \id\indices{x_{k}}{y_{k}} 
    \\
    = \phantom{+} & 0 .
    %-----------------------------------------------------------------------
    \intertext{The first factors vanish whenever $x_j \not\equiv y_j$ 
      for some $j$ with $k < j \le n$; otherwise the second factors 
      vanish because of our hypothesis $x_k \ne y_k$.
      On the other hand we have:}
    %-----------------------------------------------------------------------
    (s^{n+1}_m d^{n}_{k-1} f)\indices{x_1,\dots,x_n}{y_1,\dots,y_n} =
    + & (d^n_{k-1} f)\indices%
    {x_1,\dots,x_{k-1},u,x_k,\dots,x_n}%
    {y_1,\dots,y_{k-1},v,y_k,\dots,y_n} 
    \\
    = + & f\indices%
    {x_1^{\phantom{u}},\dots,x_{k-1}^{\phantom{u}},x_{k},\dots,x_n}%
    {y_1^{\phantom{v}},\dots,y_{k-1}^{\phantom{v}},y_{k},\dots,y_n} 
    \cdot \id\indices{u^{x_{k}\cdots x_n}}{v^{y_{k}\cdots y_n}} 
    \\ 
    - & f\indices%
    {x_1^{u},\dots,x_{k-1}^{u},x_{k},\dots,x_n}% 
    {y_1^{v},\dots,y_{k-1}^{v},y_{k},\dots,y_n} 
    \cdot \id\indices{u}{v} 
    \\
    = \phantom{+} & f\indices%
    {x_1,\dots,x_{k-1},x_{k},\dots,x_n}%
    {y_1,\dots,y_{k-1},y_{k},\dots,y_n} .
    %-----------------------------------------------------------------------
    \intertext{The first factors vanish whenever $x_j \not\equiv y_j$ 
      for some $j$ with $k < j \le n$; otherwise we have $u \ne v$ 
      with $u^{x_k} = v^{y_k}$, whence $u^{x_{k}\cdots x_n} = v^{y_{k}\cdots y_n}$.
      \bigskip \newline 
      \textit{Third case.} 
      For $i \ge k$ we find:}
    %-----------------------------------------------------------------------
    (d^{n-1}_{i} s^n_m f)\indices{x_1,\dots,x_n}{y_1,\dots,y_n} =
    + & (s^n_m f)\indices%
    {x_1^{\phantom{x_{i+1}}},\dots,x_{k}^{\phantom{x_{i+1}}},\dots,x_{i}^{\phantom{x_{i+1}}},x_{i+2},\dots,x_n}%
    {y_1^{\phantom{y_{i+1}}},\dots,y_{k}^{\phantom{y_{i+1}}},\dots,y_{i}^{\phantom{y_{i+1}}},y_{i+2},\dots,y_n} 
    \cdot \id\indices{x_{i+1}^{x_{i+2}\cdots x_n}}{y_{i+1}^{y_{i+2}\cdots y_n}} 
    \\ 
    - & (s^n_m f)\indices%
    {x_1^{x_{i+1}},\dots,x_{k}^{x_{i+1}},\dots,x_{i}^{x_{i+1}},x_{i+2},\dots,x_n}% 
    {y_1^{y_{i+1}},\dots,y_{k}^{y_{i+1}},\dots,y_{i}^{y_{i+1}},y_{i+2},\dots,y_n} 
    \cdot \id\indices{x_{i+1}}{y_{i+1}} 
    \\ 
    = \phantom{+} & 0 .
    %-----------------------------------------------------------------------
    \intertext{The first summand vanishes because $x_k \not\equiv y_k$;
      the second summand vanishes because $x_{i+1} \ne y_{i+1}$ or 
      $x_{k}^{x_{i+1}} \not\equiv y_{k}^{y_{i+1}}$.  Analogously:}
    %-----------------------------------------------------------------------
    (s^{n+1}_m d^{n}_{k} f)\indices{x_1,\dots,x_n}{y_1,\dots,y_n} =
    + & (d^n_{k} f)\indices%
    {x_1,\dots,x_{k-1},u,x_k,\dots,x_n}%
    {y_1,\dots,y_{k-1},v,y_k,\dots,y_n} 
    \\
    = + & f\indices%
    {x_1^{\phantom{x_{k}}},\dots,x_{k-1}^{\phantom{x_{k}}},u^{\phantom{x_{k}}},x_{k+1},\dots,x_n}%
    {y_1^{\phantom{y_{k}}},\dots,y_{k-1}^{\phantom{y_{k}}},v^{\phantom{y_{k}}},y_{k+1},\dots,y_n} 
    \cdot \id\indices{x_k^{x_{k+1}\cdots x_n}}{y_k^{y_{k+1}\cdots y_n}} 
    \\ 
    - & f\indices%
    {x_1^{x_{k}},\dots,x_{k-1}^{x_{k}},u^{x_{k}},x_{k+1},\dots,x_n}% 
    {y_1^{y_{k}},\dots,y_{k-1}^{y_{k}},v^{y_{k}},y_{k+1},\dots,y_n} 
    \cdot \id\indices{x_k}{y_k} 
    \\ 
    = \phantom{+} & 0 .
    %-----------------------------------------------------------------------
    \intertext{The first factors vanish whenever $x_j \not\equiv y_j$ 
      for some $j$ with $k < j \le n$; otherwise the second factors 
      vanish because of our hypothesis $x_k \ne y_k$.
      The same conclusion holds for $i > k$:}
    %-----------------------------------------------------------------------
    (s^{n+1}_m d^{n}_{i} f)\indices{x_1,\dots,x_n}{y_1,\dots,y_n} =
    + & (d^n_{i} f)\indices%
    {x_1,\dots,x_{k-1},u,x_k,\dots,x_n}%
    {y_1,\dots,y_{k-1},v,y_k,\dots,y_n} 
    \\
    = + & f\indices%
    {x_1^{\phantom{x_{i}}},\dots,x_{k-1}^{\phantom{x_{i}}},u^{\phantom{x_{i}}},x_{k}^{\phantom{x_{i}}},
      \dots,x_{i-1}^{\phantom{x_{i}}},x_{i+1},\dots,x_n}%
    {y_1^{\phantom{y_{i}}},\dots,y_{k-1}^{\phantom{y_{i}}},v^{\phantom{y_{i}}},y_{k}^{\phantom{y_{i}}},
      \dots,y_{i-1}^{\phantom{y_{i}}},y_{i+1},\dots,y_n} 
    \cdot \id\indices{x_i^{x_{i+1}\cdots x_n}}{y_i^{y_{i+1}\cdots y_n}} 
    \\ 
    - & f\indices%
    {x_1^{x_{i}},\dots,x_{k-1}^{x_{i}},u^{x_{i}},x_{k}^{x_{i}},\dots,x_{i-1}^{x_{i}},x_{i+1},\dots,x_n}% 
    {y_1^{y_{i}},\dots,y_{k-1}^{y_{i}},v^{y_{i}},y_{k}^{y_{i}},\dots,y_{i-1}^{y_{i}},y_{i+1},\dots,y_n} 
    \cdot \id\indices{x_i}{y_i} 
    \\ 
    = \phantom{+} & 0 .
  \end{align*}
  The first summand vanishes because $x_k \not\equiv y_k$;
  the second summand vanishes because $x_{i} \ne y_{i}$ or 
  $x_{k}^{x_{i}} \not\equiv y_{k}^{y_{i}}$.
\end{proof}

\begin{lemma}
  The map $t^n_m$ satisfies $t^n_m f = 0$ whenever $f \in C_{m+1}^n$.
\end{lemma}

\begin{proof}
  We show that $(t^n_m f)\smallindices{x_1,\dots,x_n}{y_1,\dots,y_n} = 0$
  for all $f \in C_{m+1}^n$ and all $x_1,\dots,x_n,y_1,\dots,y_n \in Q$.
  The previous lemma resolves the case $x_k \not\equiv y_k$,
  so it suffices to consider the remaining case where $x_k \equiv y_k$. 
  By definition of $s^{n+1}_m$ we have $(s^{n+1}_m d^n_i f)%
  \smallindices{x_1,\dots,x_n}{y_1,\dots,y_n} = 0$
  because $x_k \equiv y_k$.  Likewise, for $i \le k-2$ we find:
  \begin{align*}
    (d^{n-1}_i s^n_m f)\indices{x_1,\dots,x_n}{y_1,\dots,y_n} =
    + & (s^n_m f)\indices%
    {x_1^{\phantom{x_{i+1}}},\dots,x_{i}^{\phantom{x_{i+1}}},x_{i+2},\dots,x_k,\dots,x_n}%
    {y_1^{\phantom{y_{i+1}}},\dots,y_{i}^{\phantom{y_{i+1}}},y_{i+2},\dots,y_k,\dots,y_n} 
    \cdot \id\indices{x_{i+1}^{x_{i+2}\cdots x_n}}{y_{i+1}^{y_{i+2}\cdots y_n}} 
    \\ 
    - & (s^n_m f)\indices%
    {x_1^{x_{i+1}},\dots,x_{i}^{x_{i+1}},x_{i+2},\dots,x_k,\dots,x_n}% 
    {y_1^{y_{i+1}},\dots,y_{i}^{y_{i+1}},y_{i+2},\dots,y_k,\dots,y_n} 
    \cdot \id\indices{x_{i+1}}{y_{i+1}} 
    \\ 
    = \phantom{+} & 0 .
    %-----------------------------------------------------------------------
    \intertext{For $i \ge k-1$, however, we find:}
    %-----------------------------------------------------------------------
    (d^{n-1}_{i} s^n_m f)\indices{x_1,\dots,x_n}{y_1,\dots,y_n} =
    + & (s^n_m f)\indices%
    {x_1^{\phantom{x_{i+1}}},\dots,x_{k-1}^{\phantom{x_{i+1}}},\dots,x_{i}^{\phantom{x_{i+1}}},x_{i+2},\dots,x_n}%
    {y_1^{\phantom{y_{i+1}}},\dots,y_{k-1}^{\phantom{y_{i+1}}},\dots,y_{i}^{\phantom{y_{i+1}}},y_{i+2},\dots,y_n} 
    \cdot \id\indices{x_{i+1}^{x_{i+2}\cdots x_n}}{y_{i+1}^{y_{i+2}\cdots y_n}} 
    \\ 
    - & (s^n_m f)\indices%
    {x_1^{x_{i+1}},\dots,x_{k-1}^{x_{i+1}},\dots,x_{i}^{x_{i+1}},x_{i+2},\dots,x_n}% 
    {y_1^{y_{i+1}},\dots,y_{k-1}^{y_{i+1}},\dots,y_{i}^{y_{i+1}},y_{i+2},\dots,y_n} 
    \cdot \id\indices{x_{i+1}}{y_{i+1}} .
  \end{align*}
  The summands are non-zero only if $x_{i+1} = y_{i+1}$
  and $x_j \equiv y_j$ for all $j$ with $k \le j \le n$:
  in this case their difference measures the defect of $s^n_m f$ 
  to being equivariant (jointly in the first $i$ variables).
  Both summands vanish if $x_{k-1} \equiv y_{k-1}$, 
  so let us assume $x_{k-1} \not\equiv y_{k-1}$:
  \begin{align} \label{eq:EquivarianceDefect}
    (d^{n-1}_{i} s^n_m f)\indices{x_1,\dots,x_n}{y_1,\dots,y_n} 
    = + & f\indices%
    {x_1^{\phantom{x_{i+1}}},\dots,u'\phantom{'},x_{k-1}^{\phantom{x_{i+1}}},
      \dots,x_{i}^{\phantom{x_{i+1}}},x_{i+2},\dots,x_n}%
    {y_1^{\phantom{y_{i+1}}},\dots,v'\phantom{'},y_{k-1}^{\phantom{y_{i+1}}},
      \dots,y_{i}^{\phantom{y_{i+1}}},y_{i+2},\dots,y_n} 
    % \cdot \id\indices{x_{i+1}^{x_{i+2}\cdots x_n}}{y_{i+1}^{y_{i+2}\cdots y_n}} 
    \\ \notag
    - & f\indices%
    {x_1^{x_{i+1}},\dots,u'',x_{k-1}^{x_{i+1}},\dots,x_{i}^{x_{i+1}},x_{i+2},\dots,x_n}% 
    {y_1^{y_{i+1}},\dots,v'',y_{k-1}^{y_{i+1}},\dots,y_{i}^{y_{i+1}},y_{i+2},\dots,y_n} 
    % \cdot \id\indices{x_{i+1}}{y_{i+1}} .
  \end{align}
  Here $(u',v') = \psi(x_{k-1},y_{k-1})$ and 
  $(u'',v'') = \psi(x_{k-1}^{x_{i+1}},y_{k-1}^{y_{i+1}})$.
  For $f \in C_m^n$ the contributions do in general not cancel.
  We see, however, that both summands vanish if $f \in C_{m+1}^n$.
\end{proof}

\begin{remark} \label{rem:Equivariance}
  Equation \eqref{eq:EquivarianceDefect} shows that 
  $(t^n_m f) \smallindices{x_1,\dots,x_n}{y_1,\dots,y_n}$
  can be non-zero for $f \in C_{m}^n$, 
  if $x_k \equiv y_k$ but $x_{k-1} \not\equiv y_{k-1}$.
  This equation % Equation \eqref{eq:EquivarianceDefect} 
  measures the defect of the cochain $f$,
  and our auxiliary map $\psi \colon (x_{k-1},y_{k-1}) \to (u,v)$, 
  to be equivariant under the action of $|\Inn(Q)|$.
  In the equivariant setting of \cite{Eisermann:2005}
  this defect disappears, and the projection $p^n_m$ becomes % very simple:
  % Equation \eqref{eq:EquivarianceDefect} of the previous lemma 
  % shows that whenever $s^n_m f$ is fully equivariant under the
  % action of $\Inn(Q)$, the projection $p^n_m$ becomes very simple:
  \[
  (p^n_m f)\smallindices{x_1,\dots,x_n}{y_1,\dots,y_n}
  := \begin{cases} 
    0 & \text{ if $x_j \not\equiv y_j$ for some $j$ with $n-m \le j \le n$,} 
    \\ 
    f\smallindices{x_1,\dots,x_n}{y_1,\dots,y_n} & \text{ otherwise.}
  \end{cases}
  \]
  % \[
  % (p^n_m f)\smallindices{x_1,\dots,x_n}{y_1,\dots,y_n}
  % := \begin{cases} 
  %   f\smallindices{x_1,\dots,x_n}{y_1,\dots,y_n}
  %   & \text{ if $x_j \equiv y_j$ for all $j$ with $n-m \le j \le n$,} 
  %   \\
  %   0 & \text{ otherwise.}
  % \end{cases}
  % \]
  This simplified formula has been used in \cite{Eisermann:2005}, where 
  symmetrization was applied throughout to simplify calculations.
  In our present setting we cannot apply symmetrization 
  and thus cannot assume equivariance. 
  % In the general case this trick cannot be applied, and 
  % we cannot restrict attention to the equivariant part.  
  It is remarkable, therefore, that the above calculations carry through.
  The price to pay is that % on the full module $C_m^n$ 
  the projection $p^n_m$ has a more complicated form.
\end{remark}

\subsection{Composition of homotopy retractions} \label{sub:CompositeHomotopy}

Having constructed homotopy retractions
$C_0^* \onto C_1^* \onto \dots \onto C_{m-1}^* \onto C_{m}^*$
in \sref{sub:Homotopies}, it now suffices to put the pieces together: 

\begin{corollary}
  The subcomplex $C_\Delta^*$ of quasi-diagonal cochains
  is a homotopy retract of the full Yang-Baxter cochain complex $C_\YB^*$.
  As a consequence the inclusion $C_\Delta^* \into C_\YB^*$ induces 
  an isomorphism on cohomology, $H^*(C_\Delta^*) \isoto H^*(C_\YB^*)$.
\end{corollary}

\begin{proof}
  The composition of homotopic cochain maps yields again homotopic cochain maps.  
  As a consequence, the composition of our partial homotopy retractions 
  % $C_0^* \onto C_1^* \onto \dots \onto C_{m-1}^* \onto C_{m}^*$
  yields again a homotopy retraction
  \[
  % P^*_m := p^*_{m-1} \circ \cdots \circ p^*_1 \circ p^*_0 \colon C_0^* \to C_m^* 
  P^*_m := p^*_{m-1} \circ p^*_{m-2} \circ \cdots \circ p^*_1 \circ p^*_0 \colon C_0^* \to C_m^* .
  \]
  This shows that the inclusion $C_m^* \into C_\YB^*$ is a homotopy retract.
  We wish to pass to the limit $C_\Delta^* = \bigcap_m C_m^*$.
  In each degree $n$ we have $p^n_m =  \id^n_n$ for all $m \ge n$, and thus $P^n_m = P^n_n$.
  We can thus define $P^*_\infty = \lim_{m \to \infty} P^*_m$ as the
  degree-wise limit $P^n_\infty = P^n_n$.  We conclude that
  $C_\Delta^* \into C_\YB^*$ is a homotopy retract.
\end{proof}

%%%%%%%%%%%%%%%%%%%%%%%%%%%%%%%%%%%%%%%%%%%%%%%%%%%%%%%%%%%%%%%%%%%%%%%%%%%%%

\section{From infinitesimal to complete deformations} \label{sec:CompleteDeformations}

In this section we will pass from infinitesimal to complete deformations.
In order to do so, we will assume that the ring $\A$ is complete 
with respect to the ideal $\m$, that is, we assume that 
the natural map $\A \to \varprojlim \A/\m^n$ is an isomorphism.

\begin{example} \label{exm:PowerSeriesRing}
  A polynomial ring $\K[h]$ % over a ring $\K$
  is not complete with respect to the ideal $(h)$.
  Its completion is the power series ring $\K\fps{h} = \varprojlim \K[h]/(h^n)$.
  The latter is complete with respect to its ideal $\m = (h)$.
  If $\K$ is a field, then $\K\fps{h}$ is a complete local ring,
  which means that $\m$ is the unique maximal ideal 
  and $\K\fps{h}$ is complete with respect to $\m$.
\end{example}

\begin{example} \label{exm:PadicIntegers}
  The ring of integers $\Z$ is not complete with respect to 
  the ideal $(p)$, where $p$ will be assumed to be prime.
  Its completion is the ring of $p$-adic integers
  $\Z_{p} = \varprojlim \Zmod{p^n}$.  The latter
  is complete with respect to its unique maximal ideal $\m = (p)$.
\end{example}

Completions lend themselves to induction techniques.
As the inductive step, we assume that $\m^{n+1} = 0$.
One can always force this condition by passing to the quotient $\A/\m^{n+1}$.

\begin{lemma} \label{lem:InductiveStep}
  Consider a ring $\A$ with ideal $\m$ such that $\m^{n+1} = 0$.
  Let $c \colon \A{Q^2} \to \A{Q^2}$ be a Yang-Baxter operator that 
  satisfies $c \equiv c_Q$ modulo $\m$ and is quasi-diagonal modulo $\m^n$.
  Then % $c$ is equivalent to a quasi-diagonal Yang-Baxter operator. More precisely, 
  there exists $\alpha\colon \A{Q} \to \A{Q}$ with $\alpha \equiv \id_V$ modulo $\m^n$, 
  such that $(\alpha\tensor\alpha)^{-1} \circ c \circ (\alpha\tensor\alpha)$ 
  is a quasi-diagonal deformation of $c_Q$.
\end{lemma}

\begin{proof}
  We have $c = c_Q \circ F$ with $F \equiv \id_V^{\tensor 2}$ modulo $\m$.
  We write $F$ in matrix notation as a map $F \colon Q^2 \times Q^2 \to \A$.
  Its non-quasi-diagonal part $f \colon Q^2 \times Q^2 \to \A$ is defined by 
  \[
  f\smallindices{x_1,x_2}{y_1,y_2}
  := \begin{cases} 
    0 & \text{ if $x_1 \equiv y_1$ and $x_2 \equiv y_2$,} \\
    F\smallindices{x_1,x_2}{y_1,y_2} & \text{ otherwise.}
    %% F\smallindices{x_1,x_2}{y_1,y_2}
    %% & \text{ if $x_1 \not\equiv y_1$ or $x_2 \not\equiv y_2$,} 
    %% \\ 0 & \text{ otherwise.}
  \end{cases}
  \]
  By hypothesis $f$ takes values in $\m^n \subset \A$,
  and can thus be considered as a cochain $C_\YB^2(c_Q;\m^n)$.
  The map $\bar{c} = c_Q \circ ( F - f ) = c \circ ( \id_V^{\tensor 2} - f )$
  is quasi-diagonal, by construction.  

  We claim that $\bar{c}$ is actually a Yang-Baxter operator.
  We know that $c$ satisfies the Yang-Baxter equation;
  its deformation $\bar{c}$ thus satisfies
  \begin{equation} \label{eq:SeparationTrick}
    \id_V^{\tensor 3} - 
    (\id_V \tensor \bar{c})^{-1}(\bar{c} \tensor \id_V)^{-1}
    (\id_V \tensor \bar{c})^{-1}(\bar{c} \tensor \id_V)
    (\id_V \tensor \bar{c})(\bar{c} \tensor \id_V) 
    = d^2 f .
  \end{equation}
  It is easy to check that the left-hand side is a quasi-diagonal map,
  whereas the right-hand side is zero on the quasi-diagonal.
  We conclude that \emph{both} sides must vanish. 
  This means that $\bar{c}$ satisfies the Yang-Baxter equation, 
  and that $f \in C_\YB^2(c_Q;\m^n)$ is a cocycle.
  
  By Theorem \ref{thm:QuasiDiagonalCohomology},
  the inclusion $C_\Delta^*(c_Q;\m^n) \subset C_\YB^*(c_Q;\m^n)$
  induces an isomorphism on cohomology.  The class $[f] \in C_\YB^2(c_Q;\m^n)$
  can thus be presented by a quasi-diagonal cocycle $\tilde{f} \in C_\Delta^2(c_Q;\m^n)$.
  This means that there exists a cochain $g \in C_\YB^1(c_Q;\m^n)$ such that $\tilde{f} = f + d^1 g$.
  We conclude that $\alpha = \id_V + g$ conjugates $c$ to a quasi-diagonal
  Yang-Baxter operator $\tilde{c} = (\alpha\tensor\alpha)^{-1} \circ c \circ (\alpha\tensor\alpha)$,
  as desired.
\end{proof}

\begin{remark}
  In the preceding proof the construction and analysis 
  of $\bar{c}$ serve to show that $f$ is a $2$-cocycle.  
  The separation trick of Equation \eqref{eq:SeparationTrick} 
  is taken from \cite[\textsection4]{Eisermann:2005}.
  I seize the opportunity to point out that there 
  the difference \eqref{eq:SeparationTrick} 
  is misprinted and lacks the term  $\id_V^{\tensor 3}$. 
  With this small correction the argument applies as intended.
\end{remark}

\begin{remark}
  In the proof of Lemma \ref{lem:InductiveStep}
  we do not claim that $c$ is conjugate to $\bar{c}$.
  This is true in the equivariant setting of \cite{Eisermann:2005}, 
  but without equivariance it is false in general:
  the coboundary $d^1 g$ kills the non-quasi-diagonal part 
  but usually also changes the quasi-diagonal part
  (see Remark \ref{rem:Equivariance}).
  % Our approach parallels \cite[\textsection4]{Eisermann:2005},
  % but as pointed out in Remark \ref{rem:Equivariance},
  % the lack of equivariance complicates matters: we cannot
  % project to the quasi-diagonal part in a na\"ive way.
\end{remark}

% To conclude the passage from infinitesimal to complete, 
% it only remains to put the ingredients together:

\begin{theorem}
  Let $\A$ be a ring that is complete with respect to the ideal $\m$.
  Then every Yang-Baxter deformation of $c_Q$ over $\A$ 
  is equivalent to a quasi-diagonal deformation.
\end{theorem}

% \begin{theorem}
%   Suppose that the ring $\A$ is complete with respect to the ideal $\m$.
%   Let $Q$ be a rack and let $c_Q \colon \A{Q^2} \to \A{Q^2}$
%   be the associated Yang-Baxter operator.  
%   Then every Yang-Baxter deformation of $c_Q$ over $\A$ with respect to $\m$
%   is equivalent to a quasi-diagonal deformation, that is, one of the form
%   $c_Q (\id_{\smash{V}}^{\tensor 2} + f)$ where $f \colon \A{Q^2} \to \m{Q^2}$ 
%   is quasi-diagonal.
% \end{theorem}

\begin{proof}
  Starting with $c_1 := c$ for $n=1$, suppose that $c_n = c_Q f_n$
  has a deformation term $f_n$ that is quasi-diagonal modulo $\m^n$.
  By Lemma \ref{lem:InductiveStep}, there exists $\alpha_n\colon \A{Q} \to \A{Q}$
  with $\alpha_n \equiv \id_V$ modulo $\m^n$, such that 
  $c_{n+1} := (\alpha_n \tensor \alpha_n)^{-1} c_n (\alpha_n \tensor \alpha_n)$
  is given by $c_{n+1} = c_Q f_{n+1}$ with $f_{n+1}$ quasi-diagonal modulo $\m^{n+1}$.
  The lemma ensures that such a map $\bar\alpha_n$ exists modulo $\m^{n+1}$; 
  this can be lifted to a map $\alpha_n \colon \A{Q} \to \A{Q}$, 
  which is invertible because $\A$ is complete.
  Completeness of $\A$ also ensures that we can pass to the limit and define 
  the infinite product $\alpha = \alpha_1 \alpha_2 \alpha_3 \cdots$:
  for each $n \in \N$ this product is finite modulo $\m^n$.
  By construction, $(\alpha\tensor\alpha)^{-1} \, c \, (\alpha\tensor\alpha)$ 
  is quasi-diagonal and equivalent to $c$, as desired.
\end{proof}

\begin{corollary} \label{cor:RackRigiditiy}
  % Let $\A$ be a ring that is complete with respect to the ideal $\m$.
  If $H_\YB^2(c_Q;\m/\m^2) = \m/\m^2$, then $c_Q$ is rigid over $(\A,\m)$.
\end{corollary}

\begin{proof}
  For every unit $u \in 1 + \m$ we obtain a trivially deformed Yang-Baxter 
  operator $\tilde{c} = u \cdot c_Q$.  On the cochain level this 
  corresponds to a constant multiple of the identity,
  which induces an injection $\m/\m^2 \into H_\YB^2(c_Q;\m/\m^2)$.
  If these trivial classes exhaust all cohomology classes, 
  then degree-wise elimination as in the preceding proof
  conjugates any given Yang-Baxter deformation 
  of $c_Q$ to one of the form $u \cdot c_Q$.
\end{proof}

\section{Examples and applications} \label{sec:Examples}

\begin{example}[trivial quandle] 
  Consider first a trivial quandle $Q$, with $x \ast y = x$ for all $x,y$,
  so that $c_Q = \tau$ is simply the transposition operator.
  Here our results cannot add anything new, because the Yang-Baxter 
  complex $C_\YB^*$ is trivial, i.e.\ $d f = 0$ for all $f \in C_\YB^*$.
  In particular there are no infinitesimal obstructions: \emph{every} 
  deformation of $\tau$ satisfies the Yang-Baxter equation modulo $\m^2$.
  There are, however, higher-order obstructions: these form a subject 
  of their own and belong to the much deeper theory of quantum invariants
  \cite{Drinfeld:1987,Turaev:1988,Kassel:1995,KasselRossoTuraev:1997}.
\end{example}

\begin{example}[faithful quandle]
  Next we consider the other extreme, where Theorem 
  \ref{thm:QuasiDiagonalDeformation} applies most efficiently.
  Let $G$ be a centreless group, so that conjugation
  induces an isomorphism $G \isoto \Inn(G)$.
  Suppose that $Q \subset G$ is a conjugacy class that generates $G$.
  Then we have $\Inn(Q) \cong \Inn(G) \cong G$, and the inner
  representation $\rho \colon Q \to \Inn(Q)$ is injective.
  % (This could be called a \emph{faithful quandle}.)
  In this case every Yang-Baxter deformation of $c_Q$ 
  over a complete ring $\A$ is equivalent to a diagonal deformation.
  If $|G|$ is finite and invertible in $\A$, 
  then $c_Q$ is rigid \cite{Eisermann:2005}.
\end{example}

\begin{example}[dihedral quandle of order $3$]
  % We give a concrete example where the modular case is rigid.
  The smallest non-trivial example of a rigid operator $c_Q$
  is given by the quandle $Q = \{ (12), (13), (23) \}$, formed
  by transpositions in the symmetric group $S_3$, or equivalently 
  the set of reflections of an equilateral triangle.
  % in the dihedral group $D_3$.

  Ordering the basis $Q\times Q$ lexicographically, 
  we can represent $c_Q$ by the matrix
  \[
  \newcommand{\0}{\cdot}
  c_Q = \left(\begin{smallmatrix}
       1 & \0 & \0 & \0 & \0 & \0 & \0 & \0 & \0 \\
      \0 & \0 & \0 & \0 & \0 & \0 &  1 & \0 & \0 \\
      \0 & \0 & \0 &  1 & \0 & \0 & \0 & \0 & \0 \\
      \0 & \0 & \0 & \0 & \0 & \0 & \0 &  1 & \0 \\
      \0 & \0 & \0 & \0 &  1 & \0 & \0 & \0 & \0 \\
      \0 &  1 & \0 & \0 & \0 & \0 & \0 & \0 & \0 \\
      \0 & \0 & \0 & \0 & \0 &  1 & \0 & \0 & \0 \\
      \0 & \0 &  1 & \0 & \0 & \0 & \0 & \0 & \0 \\
      \0 & \0 & \0 & \0 & \0 & \0 & \0 & \0 &  1
    \end{smallmatrix}\right).
  \]

  In the quantum case, the initial operator $\tau$ 
  is trivial but its deformations are highly interesting.
  In the present example, the interesting part 
  is the initial operator $c_Q$ itself:
  the associated link invariant is the number 
  of $3$-colourings, as defined by R.H.\,Fox \cite{Fox:1962,Fox:1970}.

  Unlike $\tau$, the operator $c_Q$ does not admit 
  any non-trivial deformation over $\Q\fps{h}$.
  In this sense it is an isolated solution of the Yang-Baxter equation.
  We can now prove more:
\end{example}

\begin{proposition}
  For the quandle $Q = \{ (12), (13), (23) \} \subset S_3$ the associated 
  Yang-Baxter operator $c_Q$ is rigid over every complete ring.
\end{proposition}

\begin{proof}
  According to \cite{Eisermann:2005}, the operator $c_Q$ is rigid 
  over every ring $\A$ in which the order $|S_3|=6$ is invertible.
  Potentially there could exist non-trivial deformations 
  in characteristic $2$ or $3$. % over $\Zmod{2}$ or $\Zmod{3}$.
  Theorem \ref{thm:QuasiDiagonalCohomology} ensures that 
  infinitesimal deformations are quasi-diagonal, 
  which means diagonal in the present example 
  because $\rho \colon Q \to \Inn(Q) = S_3$ is injective.
  According to Proposition \ref{prop:RackYBCohomology},
  diagonal deformations correspond to rack cohomology.
  A direct calculation shows that $H_\Rack^2(Q;\m) \cong \m$ 
  for all modules $\m$, see \cite{CarterEtAl:2003}, whence
  Corollary \ref{cor:RackRigiditiy} implies rigidity.
  % The corresponding deformations simply multiply $c_Q$ by a constant,
  % so that $c_Q$ is rigid also in the modular case.
\end{proof}

\begin{example}[the other quandle of order $3$]
  \newcommand{\la}[1]{\lambda_{#1}}

  The smallest quandle that is non-trivial yet deformable is 
  $Q = \{a,b,c\}$ with operation given by the table below.
  Ordering the basis $Q\times Q$ lexicographically, 
  we obtain the matrix of $c_Q$ as indicated.
  We have $\Inn(Q) = \Zmod{2}$:  if $2$ is invertible in $\K$, 
  then $H_\YB^2(c_Q;\K)$ is free of rang $9$ and can easily 
  be made explicit using the results of \cite{Eisermann:2005}.
  We state it here in form of a $9$-parameter % infinitesimal 
  deformation $c = c_Q \circ ( \id_V^{\tensor 2} + f )$, where $f \in C_\YB^2(c_Q,\m)$
  is quasi-diagonal and equivariant under $\Inn(Q) \times \Inn(Q)$:
  \[
  Q = \left\{
  \begin{tabular}{c|ccc}
    $\ast$ & $a$ & $b$ & $c$ \\
    \hline
    $a$    & $a$ & $a$ & $b$ \\
    $b$    & $b$ & $b$ & $a$ \\
    $c$    & $c$ & $c$ & $c$
  \end{tabular} \right. ,
  \quad
  \newcommand{\0}{\cdot}
  c_Q = \left(\begin{smallmatrix}
       1 & \0 & \0 & \0 & \0 & \0 & \0 & \0 & \0 \\
      \0 & \0 & \0 &  1 & \0 & \0 & \0 & \0 & \0 \\
      \0 & \0 & \0 & \0 & \0 & \0 &  1 & \0 & \0 \\
      \0 &  1 & \0 & \0 & \0 & \0 & \0 & \0 & \0 \\
      \0 & \0 & \0 & \0 &  1 & \0 & \0 & \0 & \0 \\
      \0 & \0 & \0 & \0 & \0 & \0 & \0 &  1 & \0 \\
      \0 & \0 & \0 & \0 & \0 &  1 & \0 & \0 & \0 \\
      \0 & \0 &  1 & \0 & \0 & \0 & \0 & \0 & \0 \\
      \0 & \0 & \0 & \0 & \0 & \0 & \0 & \0 &  1
    \end{smallmatrix}\right) ,
  \quad
  f = \left(\begin{smallmatrix}
      \la{1} & \la{2} & \0     & \la{3} & \la{4} & \0     & \0     & \0     & \0     \\
      \la{2} & \la{1} & \0     & \la{4} & \la{3} & \0     & \0     & \0     & \0     \\
      \0     & \0     & \la{5} & \0     & \0     & \la{6} & \0     & \0     & \0     \\
      \la{3} & \la{4} & \0     & \la{1} & \la{2} & \0     & \0     & \0     & \0     \\
      \la{4} & \la{3} & \0     & \la{2} & \la{1} & \0     & \0     & \0     & \0     \\
      \0     & \0     & \la{6} & \0     & \0     & \la{5} & \0     & \0     & \0     \\
      \0     & \0     & \0     & \0     & \0     & \0     & \la{7} & \la{8} & \0     \\
      \0     & \0     & \0     & \0     & \0     & \0     & \la{8} & \la{7} & \0     \\
      \0     & \0     & \0     & \0     & \0     & \0     & \0     & \0     & \la{9}   
    \end{smallmatrix}\right) .
  \]

  We remark that the deformed operator $c$ 
  satisfies the Yang-Baxter equation to all orders,
  and not only infinitesimally modulo $\m^2$.

  A priori there could exist more deformations over $\Zmod{2}$,
  but a computer calculation shows that $\dim H_\YB^2(c_Q;\Zmod{2}) = 9$.
  So there are no additional deformations in the modular case.
\end{example}

% The deformed operator $c(\lambda) = c_Q \circ ( \id_V^{\tensor 2} + f )$ satisfies
% $\tr(c^2) = 4(1+\lambda_1)^2+8\lambda_2\lambda_3+4\lambda_4^2+4(1+\lambda_7)\lambda_6+4\lambda_8(1+\lambda_5)+(1+\lambda_9)^2$.
% This shows that none of the parameters can be eliminated by a gauge transformations.

% It is remarkable that the above deformation satisfies the Yang-Baxter equation
% not only infinitesimally but to any given order.

\begin{example}[dihedral quandle of order $4$]
  \newcommand{\la}[1]{\lambda_{#1}}
  \newcommand{\pa}{\alpha} \newcommand{\pb}{\beta} \newcommand{\pc}{\gamma} \newcommand{\pd}{\delta}

  There exist quandles for which the modular case 
  offers more Yang-Baxter deformations than the rational case.
  We wish to illustrate this by an example where 
  the additional deformations are not diagonal but quasi-diagonal.
  % There exist quandles that posses more modular 
  % Yang-Baxter deformations than rational deformations.
  The smallest such example is given by the set of reflections of a square, 
  \[
  Q=\{\; (13)\,,\; (24)\,,\; (12)(34)\,,\; (14)(23) \;\} .
  \]
  This set is closed under conjugation 
  in the symmetric group $S_4$, hence a quandle.
  With respect to the lexicographical basis, $c_Q$ is 
  represented by the following permutation matrix:
  \[
  \newcommand{\0}{\cdot}
  c_Q = 
  \left( \begin{smallmatrix}
       1 & \0 & \0 & \0 & \0 & \0 & \0 & \0 & \0 & \0 & \0 & \0 & \0 & \0 & \0 & \0 \\
      \0 & \0 & \0 & \0 &  1 & \0 & \0 & \0 & \0 & \0 & \0 & \0 & \0 & \0 & \0 & \0 \\
      \0 & \0 & \0 & \0 & \0 & \0 & \0 & \0 & \0 & \0 & \0 & \0 &  1 & \0 & \0 & \0 \\
      \0 & \0 & \0 & \0 & \0 & \0 & \0 & \0 &  1 & \0 & \0 & \0 & \0 & \0 & \0 & \0 \\
      \0 &  1 & \0 & \0 & \0 & \0 & \0 & \0 & \0 & \0 & \0 & \0 & \0 & \0 & \0 & \0 \\
      \0 & \0 & \0 & \0 & \0 &  1 & \0 & \0 & \0 & \0 & \0 & \0 & \0 & \0 & \0 & \0 \\
      \0 & \0 & \0 & \0 & \0 & \0 & \0 & \0 & \0 & \0 & \0 & \0 & \0 &  1 & \0 & \0 \\
      \0 & \0 & \0 & \0 & \0 & \0 & \0 & \0 & \0 &  1 & \0 & \0 & \0 & \0 & \0 & \0 \\
      \0 & \0 & \0 & \0 & \0 & \0 &  1 & \0 & \0 & \0 & \0 & \0 & \0 & \0 & \0 & \0 \\
      \0 & \0 &  1 & \0 & \0 & \0 & \0 & \0 & \0 & \0 & \0 & \0 & \0 & \0 & \0 & \0 \\
      \0 & \0 & \0 & \0 & \0 & \0 & \0 & \0 & \0 & \0 &  1 & \0 & \0 & \0 & \0 & \0 \\
      \0 & \0 & \0 & \0 & \0 & \0 & \0 & \0 & \0 & \0 & \0 & \0 & \0 & \0 &  1 & \0 \\
      \0 & \0 & \0 & \0 & \0 & \0 & \0 &  1 & \0 & \0 & \0 & \0 & \0 & \0 & \0 & \0 \\
      \0 & \0 & \0 &  1 & \0 & \0 & \0 & \0 & \0 & \0 & \0 & \0 & \0 & \0 & \0 & \0 \\
      \0 & \0 & \0 & \0 & \0 & \0 & \0 & \0 & \0 & \0 & \0 &  1 & \0 & \0 & \0 & \0 \\
      \0 & \0 & \0 & \0 & \0 & \0 & \0 & \0 & \0 & \0 & \0 & \0 & \0 & \0 & \0 &  1 \\
    \end{smallmatrix} \right) . 
  % 
  % The following transposed matrix has been erroneously been printed before
  % \left( \begin{smallmatrix}
  %     1 & \0 & \0 & \0 & \0 & \0 & \0 & \0 & \0 & \0 & \0 & \0 & \0 & \0 & \0 & \0 \\
  %     \0 & \0 & \0 & \0 &  1 & \0 & \0 & \0 & \0 & \0 & \0 & \0 & \0 & \0 & \0 & \0 \\
  %     \0 & \0 & \0 & \0 & \0 & \0 & \0 & \0 & \0 &  1 & \0 & \0 & \0 & \0 & \0 & \0 \\
  %     \0 & \0 & \0 & \0 & \0 & \0 & \0 & \0 & \0 & \0 & \0 & \0 & \0 &  1 & \0 & \0 \\
  %     \0 &  1 & \0 & \0 & \0 & \0 & \0 & \0 & \0 & \0 & \0 & \0 & \0 & \0 & \0 & \0 \\
  %     \0 & \0 & \0 & \0 & \0 &  1 & \0 & \0 & \0 & \0 & \0 & \0 & \0 & \0 & \0 & \0 \\
  %     \0 & \0 & \0 & \0 & \0 & \0 & \0 & \0 &  1 & \0 & \0 & \0 & \0 & \0 & \0 & \0 \\
  %     \0 & \0 & \0 & \0 & \0 & \0 & \0 & \0 & \0 & \0 & \0 & \0 &  1 & \0 & \0 & \0 \\
  %     \0 & \0 & \0 &  1 & \0 & \0 & \0 & \0 & \0 & \0 & \0 & \0 & \0 & \0 & \0 & \0 \\
  %     \0 & \0 & \0 & \0 & \0 & \0 & \0 &  1 & \0 & \0 & \0 & \0 & \0 & \0 & \0 & \0 \\
  %     \0 & \0 & \0 & \0 & \0 & \0 & \0 & \0 & \0 & \0 &  1 & \0 & \0 & \0 & \0 & \0 \\
  %     \0 & \0 & \0 & \0 & \0 & \0 & \0 & \0 & \0 & \0 & \0 & \0 & \0 & \0 &  1 & \0 \\
  %     \0 & \0 &  1 & \0 & \0 & \0 & \0 & \0 & \0 & \0 & \0 & \0 & \0 & \0 & \0 & \0 \\
  %     \0 & \0 & \0 & \0 & \0 & \0 &  1 & \0 & \0 & \0 & \0 & \0 & \0 & \0 & \0 & \0 \\
  %     \0 & \0 & \0 & \0 & \0 & \0 & \0 & \0 & \0 & \0 & \0 &  1 & \0 & \0 & \0 & \0 \\
  %     \0 & \0 & \0 & \0 & \0 & \0 & \0 & \0 & \0 & \0 & \0 & \0 & \0 & \0 & \0 &  1 \\
  %   \end{smallmatrix} \right) . 
  \]

  By construction, this matrix is a solution 
  of the Yang-Baxter equation over any ring $\A$.
  According to \cite{Eisermann:2005} it admits a $16$-parameter deformation 
  $c(\lambda) = c_Q \circ ( \id_V^{\tensor 2} + f )$ given by the following matrix,
  which is quasi-diagonal and equivariant under $\Inn(Q) \times \Inn(Q)$:
  \[
  \newcommand{\0}{\cdot}
  f =   
  \left( \begin{smallmatrix}
      \la{ 1}&\la{ 2}&\0&\0&\la{ 3}&\la{ 4}&\0&\0&\0&\0&\0&\0&\0&\0&\0&\0\\
      \la{ 2}&\la{ 1}&\0&\0&\la{ 4}&\la{ 3}&\0&\0&\0&\0&\0&\0&\0&\0&\0&\0\\
      \0&\0&\la{ 5}&\la{ 6}&\0&\0&\la{ 7}&\la{ 8}&\0&\0&\0&\0&\0&\0&\0&\0\\
      \0&\0&\la{ 6}&\la{ 5}&\0&\0&\la{ 8}&\la{ 7}&\0&\0&\0&\0&\0&\0&\0&\0\\
      \la{ 3}&\la{ 4}&\0&\0&\la{ 1}&\la{ 2}&\0&\0&\0&\0&\0&\0&\0&\0&\0&\0\\
      \la{ 4}&\la{ 3}&\0&\0&\la{ 2}&\la{ 1}&\0&\0&\0&\0&\0&\0&\0&\0&\0&\0\\
      \0&\0&\la{ 7}&\la{ 8}&\0&\0&\la{ 5}&\la{ 6}&\0&\0&\0&\0&\0&\0&\0&\0\\
      \0&\0&\la{ 8}&\la{ 7}&\0&\0&\la{ 6}&\la{ 5}&\0&\0&\0&\0&\0&\0&\0&\0\\
      \0&\0&\0&\0&\0&\0&\0&\0&\la{ 9}&\la{10}&\0&\0&\la{11}&\la{12}&\0&\0\\
      \0&\0&\0&\0&\0&\0&\0&\0&\la{10}&\la{ 9}&\0&\0&\la{12}&\la{11}&\0&\0\\
      \0&\0&\0&\0&\0&\0&\0&\0&\0&\0&\la{13}&\la{14}&\0&\0&\la{15}&\la{16}\\
      \0&\0&\0&\0&\0&\0&\0&\0&\0&\0&\la{14}&\la{13}&\0&\0&\la{16}&\la{15}\\
      \0&\0&\0&\0&\0&\0&\0&\0&\la{11}&\la{12}&\0&\0&\la{ 9}&\la{10}&\0&\0\\
      \0&\0&\0&\0&\0&\0&\0&\0&\la{12}&\la{11}&\0&\0&\la{10}&\la{ 9}&\0&\0\\
      \0&\0&\0&\0&\0&\0&\0&\0&\0&\0&\la{15}&\la{16}&\0&\0&\la{13}&\la{14}\\
      \0&\0&\0&\0&\0&\0&\0&\0&\0&\0&\la{16}&\la{15}&\0&\0&\la{14}&\la{13}
    \end{smallmatrix} \right).   
  \]

  For every choice of parameters $\la{1},\dots,\la{16} \in \m$ 
  the matrix $c(\lambda)$ satisfies the Yang-Baxter equation
  (to all orders) and thus deforms $c(0)=c_Q$ over $(\A,\m)$.
  % We also remark that 
  % \begin{align*}
  %   \tr \left[ c(\lambda)^2 \right] 
  %   & = 4(\la{ 1}+1)^2 + 4\la{ 4}^2 + 4(\la{13}+1)^2 + 4\la{16}^2 \\
  %   & + 8(\la{ 9}+1)\la{ 8} + 8(\la{ 5}+1)\la{12}
  %   + 8\la{ 2}\la{ 3} + 8\la{14}\la{15}
  %   + 8\la{10}\la{ 6} + 8\la{11}\la{ 7} .
  % \end{align*}
  % This shows that none of the parameters 
  % can be eliminated by a gauge transformation.
  
  The quandle $Q$ has the inner automorphism group 
  $\Inn(Q) \cong \Zmod{2} \times \Zmod{2}$, of order $4$.
  If $2$ is invertible in $\K$, then $H_\YB^*(c_Q;\K)$ can be calculated 
  using the results of \cite{Eisermann:2005} and is easily seen 
  to be free of rank $16$ such that $f$ is the most general deformation.
  In particular we have $\dim H_\YB^*(c_Q;\K)=16$ for every field $\K$ 
  of characteristic $\ne 2$. % $\operatorname{char} \K \ne 2$.

  Over $\K = \Zmod{2}$, however, a computer calculation 
  shows that $\dim H_\YB^2(c_Q;\Zmod{2}) = 20$, which means 
  that there exists a $20$-parameter deformation, at least infinitesimally.
  We state the result in the form $c = c_Q( \id_V^{\tensor 2} + f + g )$ as follows.

  First we have the $16$-parameter family that appears in every characteristic:
  \[
  \newcommand{\0}{\cdot}
  f =   
  \left( \begin{smallmatrix}
      \la{ 1}&\la{ 2}&\0&\0&\la{ 3}&\la{ 4}&\0&\0&\0&\0&\0&\0&\0&\0&\0&\0\\
      \la{ 2}&\la{ 1}&\0&\0&\la{ 4}&\la{ 3}&\0&\0&\0&\0&\0&\0&\0&\0&\0&\0\\
      \0&\0&\la{ 5}'&\la{ 6}&\0&\0&\la{ 7}'&\la{ 8}&\0&\0&\0&\0&\0&\0&\0&\0\\
      \0&\0&\la{ 6}&\la{ 5}'&\0&\0&\la{ 8}&\la{ 7}'&\0&\0&\0&\0&\0&\0&\0&\0\\
      \la{ 3}&\la{ 4}&\0&\0&\la{ 1}&\la{ 2}&\0&\0&\0&\0&\0&\0&\0&\0&\0&\0\\
      \la{ 4}&\la{ 3}&\0&\0&\la{ 2}&\la{ 1}&\0&\0&\0&\0&\0&\0&\0&\0&\0&\0\\
      \0&\0&\la{ 7}''&\la{ 8}&\0&\0&\la{ 5}''&\la{ 6}&\0&\0&\0&\0&\0&\0&\0&\0\\
      \0&\0&\la{ 8}&\la{ 7}''&\0&\0&\la{ 6}&\la{ 5}''&\0&\0&\0&\0&\0&\0&\0&\0\\
      \0&\0&\0&\0&\0&\0&\0&\0&\la{ 9}'&\la{10}&\0&\0&\la{11}'&\la{12}&\0&\0\\
      \0&\0&\0&\0&\0&\0&\0&\0&\la{10}&\la{ 9}'&\0&\0&\la{12}&\la{11}'&\0&\0\\
      \0&\0&\0&\0&\0&\0&\0&\0&\0&\0&\la{13}&\la{14}&\0&\0&\la{15}&\la{16}\\
      \0&\0&\0&\0&\0&\0&\0&\0&\0&\0&\la{14}&\la{13}&\0&\0&\la{16}&\la{15}\\
      \0&\0&\0&\0&\0&\0&\0&\0&\la{11}''&\la{12}&\0&\0&\la{ 9}''&\la{10}&\0&\0\\
      \0&\0&\0&\0&\0&\0&\0&\0&\la{12}&\la{11}''&\0&\0&\la{10}&\la{ 9}''&\0&\0\\
      \0&\0&\0&\0&\0&\0&\0&\0&\0&\0&\la{15}&\la{16}&\0&\0&\la{13}&\la{14}\\
      \0&\0&\0&\0&\0&\0&\0&\0&\0&\0&\la{16}&\la{15}&\0&\0&\la{14}&\la{13}
    \end{smallmatrix} \right).   
  \]
  
  For every choice of parameters in $\m$ % $\lambda_1,\lambda_2,\dots \in \m$ 
  the matrix $c(\lambda) = c_Q \circ ( \id_V^{\tensor 2} + f )$ satisfies the Yang-Baxter
  equation modulo $\m^2$.  It even satisfies the Yang-Baxter equation
  to any order provided that $\la{5}'=\la{5}''$, \; $\la{7}'=\la{7}''$, \;  
  $\la{9}'=\la{9}''$, \; $\la{11}'=\la{11}''$. 

  We set $\la{5}=\la{5}'+\la{5}''$, \; $\la{7}=\la{7}'+\la{7}''$, \;  
  $\la{9}=\la{9}'+\la{9}''$, \; $\la{11}=\la{11}'+\la{11}''$. 
  Two deformations $c(\lambda)$ and $c(\tilde\lambda)$ are gauge equivalent 
  if and only if $\lambda_k = \tilde\lambda_k$ for all $k=1,\dots,16$.
  We have chosen the redundant formulation above in order 
  to highlight the symmetry resp.\ the symmetry breaking.
  If $2$ were invertible we would simply set 
  $\la{5}'=\la{5}''=\frac{1}{2}\la{5}$ etc.
  In characteristic $2$ we can realize $\la{5}=1$ either by 
  $\la{5}'=1$ and $\la{5}''=0$, or by $\la{5}'=0$ and $\la{5}''=1$.
  Both deformations are gauge equivalent, but no symmetric form is possible.

  Next we have $4$-parameters deformation that appears only in characteristic $2$:
  \[
  \newcommand{\0}{\cdot}
  g =   
  \left( \begin{smallmatrix}
      \pa' &\0   &\0   &\0   &\pb' &\0   &\0   &\0   &\0   &\0   &\0   &\0   &\0   &\0   &\0   &\0   \\
      \0   &\pa''&\0   &\0   &\0   &\pb''&\0   &\0   &\0   &\0   &\0   &\0   &\0   &\0   &\0   &\0   \\
      \0   &\0   &\pa' &\0   &\0   &\0   &\pb' &\0   &\0   &\0   &\0   &\0   &\0   &\0   &\0   &\0   \\
      \0   &\0   &\0   &\pa''&\0   &\0   &\0   &\pb''&\0   &\0   &\0   &\0   &\0   &\0   &\0   &\0   \\
      \pb''&\0   &\0   &\0   &\pa''&\0   &\0   &\0   &\0   &\0   &\0   &\0   &\0   &\0   &\0   &\0   \\
      \0   &\pb' &\0   &\0   &\0   &\pa' &\0   &\0   &\0   &\0   &\0   &\0   &\0   &\0   &\0   &\0   \\
      \0   &\0   &\pb''&\0   &\0   &\0   &\pa''&\0   &\0   &\0   &\0   &\0   &\0   &\0   &\0   &\0   \\
      \0   &\0   &\0   &\pb' &\0   &\0   &\0   &\pa' &\0   &\0   &\0   &\0   &\0   &\0   &\0   &\0   \\
      \0   &\0   &\0   &\0   &\0   &\0   &\0   &\0   &\pc' &\0   &\0   &\0   &\pd' &\0   &\0   &\0   \\
      \0   &\0   &\0   &\0   &\0   &\0   &\0   &\0   &\0   &\pc''&\0   &\0   &\0   &\pd''&\0   &\0   \\
      \0   &\0   &\0   &\0   &\0   &\0   &\0   &\0   &\0   &\0   &\pc' &\0   &\0   &\0   &\pd' &\0   \\
      \0   &\0   &\0   &\0   &\0   &\0   &\0   &\0   &\0   &\0   &\0   &\pc''&\0   &\0   &\0   &\pd''\\
      \0   &\0   &\0   &\0   &\0   &\0   &\0   &\0   &\pd''&\0   &\0   &\0   &\pc''&\0   &\0   &\0   \\
      \0   &\0   &\0   &\0   &\0   &\0   &\0   &\0   &\0   &\pd' &\0   &\0   &\0   &\pc' &\0   &\0   \\
      \0   &\0   &\0   &\0   &\0   &\0   &\0   &\0   &\0   &\0   &\pd''&\0   &\0   &\0   &\pc''&\0   \\
      \0   &\0   &\0   &\0   &\0   &\0   &\0   &\0   &\0   &\0   &\0   &\pd' &\0   &\0   &\0   &\pc' \\
    \end{smallmatrix} \right).   
  \]

  For every choice of parameters in $\m$ the matrix $c_Q \circ ( \id_V^{\tensor 2} + g )$ 
  satisfies the Yang-Baxter equation modulo $\m^2$.  
  Two such deformations are gauge equivalent 
  if and only if they share the same values
  $\pa=\pa'+\pa''$, \; $\pb=\pb'+\pb''$, \; $\pc=\pc'+\pc''$, \; $\pd=\pd'+\pd''$.
  % This means that there are four additional deformations in the modular case.
  They satisfy the Yang-Baxter equation to order $\m^3$ if and only if 
  $\pa'=\pa''$, \; $\pb'=\pb''$, \; $\pc'=\pc''$, \; $\pd'=\pd''$.
  % The implications of this example still have to be worked out.
  % % Consider complete deformations over $\Zmod{2}\fps{h}$.
  % % Consider complete deformations over $Z_{2}$.
\end{example}

\begin{example}[colouring polynomials] \label{exm:A5}
  We conclude with an example where the modular case 
  provides non-trivial diagonal deformations and 
  interesting knot invariants arise at the infinitesimal level.
  
  Consider the alternating group $G = A_5$ and 
  the conjugacy class $Q = (12345)^G$ of order $12$.  
  The knot invariant associated to $c_Q$ counts 
  for each knot $K \subset \mathbb{S}^3$ the number of 
  knot group representations $\pi_1(\mathbb{S}^3 \minus K) \to G$
  sending meridians of $K$ to elements of $Q$.
  According to \cite{Eisermann:2005} 
  the operator $c_Q$ has only trivial deformations
  over $\Q\fps{h}$ or any ring $\A$ with $5! \in \A^\times$.

  The modular case is more interesting:  
  if we consider $\A = \Z_5[h]/(h^2)$,
  then $c_Q$ does allow non-trivial deformations
  that are topologically interesting \cite[Exm.\,1.3]{Eisermann:2007}.
  According to Theorem \ref{thm:QuasiDiagonalCohomology},
  \emph{all} infinitesimal deformations of $c_Q$ are encoded by rack cohomology, 
  which has been intensely studied in recent years and is fairly well understood.
  The associated knot invariants can be identified as \emph{colouring polynomials},
  counting knot group representations $\pi_1(\mathbb{S}^3 \minus K) \to G$
  while keeping track of longitudinal information \cite{Eisermann:2007}.
\end{example}

%%%%%%%%%%%%%%%%%%%%%%%%%%%%%%%%%%%%%%%%%%%%%%%%%%%%%%%%%%%%%%%%%%%%%%%%%%%%%

% \subsection{Quandles and link invariants}

% \subsection{Rack extensions}
% Let us consider the special case of a conjugacy class $Q \subset G$ in a group $G$.

%%%%%%%%%%%%%%%%%%%%%%%%%%%%%%%%%%%%%%%%%%%%%%%%%%%%%%%%%%%%%%%%%%%%%%%%%%%%% 

\section{Open questions} \label{sec:OpenQuestions}

\subsection{From infinitesimal to complete deformations}

As explained in \sref{sec:Diagonal}, rack cohomology 
$H_\Rack^2(Q;\K)$ encodes infinitesimal deformations of $c_Q$,
that is, deformations over $\A = \K[h]/(h^2)$.  
Even at the infinitesimal level this approach 
leads to interesting knot invariants, 
as illustrated by Example \ref{exm:A5} above.
In the framework of Yang-Baxter deformations,
the following generalization appears natural:

\begin{question}
  What can be said about complete deformations,
  that is, deformations of $c_Q$ 
  over the power series ring $\K\fps{h}$
  or the $p$-adic integers $\Z_{p}$?
\end{question}

Higher order obstructions are encoded in 
$H_\YB^3(c_Q,\m^2/\m^3)$, and in general seem to be non-trivial.
For deformations over $\Q\fps{h}$ this question has been completely 
solved in \cite{Eisermann:2005}.  The modular case 
is still open and potentially more interesting.

\begin{question}
  Given a deformation of $c_Q$,
  % Given an infinitesimal or complete deformation of $c_Q$,
  what sort of topological information is contained 
  in the associated knot invariant?
\end{question}

For knot invariants coming from rack or quandle cohomology,
this question was answered in \cite{Eisermann:2007}.
For non-diagonal deformations the question is still open.
Notice that the problem gets more complicated 
and more intriguing as we approach the quantum case:
the closer $Q$ is to the trivial quandle, the more deformations will appear.
Their topological interpretation, however, becomes more difficult, 
and for the time being remains mysterious.

\subsection{From racks to biracks} % general set-theoretic solutions

Given a set $Q$ and a bijective map $c \colon Q \times Q \to Q \times Q$,
we can formulate the set-theoretic Yang-Baxter equation \cite{Drinfeld:1990} as
\[
(\id \times c)(c \times \id)(\id \times c) 
= (c \times \id)(\id \times c)(c \times \id) .
\]
% as maps $Q \times Q \times Q \to Q \times Q \times Q$.

In general $c$ will have the form $c(x,y) = ( x \rhd y, x \lhd y )$
with two binary operations $\rhd,\lhd \colon Q \times Q \to Q$, 
see \cite{ESS:1999,LYZ:2000} for details.
Recently, Kauffman's theory of virtual knots \cite{Kauffman:1999}
has rekindled interest in such set-theoretic solutions $(Q,\rhd,\lhd)$ called 
\emph{biracks} or \emph{biquandles} \cite{FJK:2004,KauffmanManturov:2005}.
Racks correspond to the case where the operation $x \rhd y = y$ 
is trivial whereas $x \lhd y = x^y$ is the rack operation.

\begin{question}
  Can our results be extended to set-theoretic solutions
  of the Yang-Baxter equation that do not come from racks?
\end{question}

Our notion of Yang-Baxter cohomology \cite{Eisermann:2005} 
has been conceived for arbitrary Yang-Baxter operators, 
and in particular it covers set-theoretic solutions 
such as biracks and biquandles above.  The restricted 
setting of diagonal deformations has been studied 
by Carter et al.\ \cite{CarterEtAl:2004}.
More general deformations still need to be examined.

%%%%%%%%%%%%%%%%%%%%%%%%%%%%%%%%%%%%%%%%%%%%%%%%%%%%%%%%%%%%%%%%%%%%%%%%%%%%%

%\bibliographystyle{plain}
\bibliographystyle{amsplain}
\bibliography{yabacohom}

%%%%%%%%%%%%%%%%%%%%%%%%%%%%%%%%%%%%%%%%%%%%%%%%%%%%%%%%%%%%%%%%%%%%%%%%%%%%%
\end{document}